\documentclass[10pt,reqno]{amsart}

\usepackage{amssymb, a4wide, amsmath, mathtools,xy}
\xyoption{all}
\usepackage{latexsym,pdfsync,xcolor,graphicx}
\usepackage{multirow}
\usepackage[T1]{fontenc}
\usepackage[utf8]{inputenc}
\usepackage{array}
\usepackage{upgreek}
\usepackage{pstricks-add}
\usepackage{enumerate}
\usepackage{orcidlink}
\usepackage[english]{babel}	
\usepackage{comment}
\allowdisplaybreaks		
\usepackage{pgf,tikz}
\usepackage{wrapfig}
\usetikzlibrary{arrows}
\usepackage{float}
\usepackage{appendix}
\usepackage{hyperref}
\hypersetup{colorlinks,
citecolor=black,
filecolor=black,
linkcolor=black,
urlcolor=black}

\newtheorem{theorem}{Theorem}[subsection]
\newtheorem{Lemma}[theorem]{Lemma}

\newtheorem{Definition}[theorem]{Definition}
\newtheorem{Corollary}[theorem]{Corollary}
\newtheorem{Remark}[theorem]{Remark}
\newtheorem{Example}[theorem]{Example}

\makeatletter
\renewcommand{\thetheorem}{%
	\ifnum\value{subsection}=0
	\thesection.\arabic{theorem}%
	\else
	\thesubsection.\arabic{theorem}%
	\fi}
\makeatother

\newcommand{\ot}{\otimes}
\newcommand{\co}{\circ}

\title[Twist deformations for Hopf coquasigroups]{Twist deformations for Hopf coquasigroups}

\author[R. González Rodríguez]{Ramón González Rodríguez
\orcidlink{0000-0003-3061-6685}}
\author[B. Ramos Pérez]{Brais Ramos Pérez
\orcidlink{0009-0006-3912-4483}}

\address[R. González Rodríguez]{CITMAga, 15782 Santiago de Compostela, Spain}
\address[R. González Rodríguez]{Departamento de Matemática Aplicada II, Universidade de Vigo, E.E. Telecomunicación, 36310 Vigo, Spain}
\email{rgon@dma.uvigo.es}
\address[B. Ramos Pérez]{Departamento de Matemáticas, Universidade de Santiago de Compostela, Facultade de Matemáticas, 15782 Santiago de Compostela, Spain}
\email{braisramos.perez@usc.es}



\begin{document}
	
\begin{abstract}
 
 In this paper, we develop a general theory of twist deformations for Hopf coquasigroups in a symmetric monoidal category. To this end, we first introduce and study non-coassociative bimonoids endowed with left and right codivisions, and establish their connection with left and right Hopf coquasigroups. Next, motivated by the classical theory of Drinfeld twists for Hopf algebras, we define twists for non-coassociative bimonoids and prove that they induce deformations of Hopf coquasigroup structures through suitable modifications of the coproduct. In particular, we obtain explicit deformation procedures for right and left Hopf coquasigroups and analyze the corresponding antipodes. Finally, we apply the general theory to construct nontrivial examples arising from Hopf coquasigroups associated with the sphere  ${\sf S}^7$, obtaining new examples of twisted Hopf coquasigroups that are neither commutative nor cocommutative.
 \end{abstract}

\maketitle


\section*{Introduction}

Hopf quasigroups and Hopf coquasigroups were introduced by J. Klim and S. Majid in \cite{KM} as non-associative and non-coassociative generalizations of Hopf algebras arising naturally in connection  with Moufang loops, composition algebras, and Mal'tsev algebras. In contrast with the classical Hopf algebra setting, the product and the coproduct  are no longer associative or coassociative, and the role of the antipode is replaced by suitable quasigroup-type identities. 

One of the central mechanisms in the theory of Hopf algebras is the deformation procedure induced by Drinfeld twists (see \cite{VD}). Such twists allow one to deform the coalgebra structure while preserving the algebra structure, leading to new and highly nontrivial examples of Hopf algebras and quantum groups. Twist deformations play a fundamental role in quantum group theory, braided categories, deformation quantization, and mathematical physics. It is therefore natural to ask whether an analogous deformation theory can be developed in the non-coassocitive framework of Hopf coquasigroups.

However, extending twist deformation techniques to the non-coassociative setting presents substantial difficulties. In the classical coassociative framework, many arguments rely heavily on coassociativity and on the compatibility properties of Galois morphisms and antipodes. In the absence of coassociativity, these techniques no longer apply directly, and one must carefully reformulate the deformation procedure in terms of (co)divisions and generalized antipode identities. This requires the development of new structural tools adapted to non-coassociative bimonoids.

The main purpose of this paper is to develop a general theory of twist deformations for Hopf coquasigroups in a strict symmetric monoidal category ${\sf C}$. Our approach is based on the study of non-coassociative bimonoids endowed with left and right codivisions, which provide a natural categorical framework for understanding left and right Hopf coquasigroup structures. We establish several structural properties of these codivisions and prove characterizations of Hopf coquasigroups in terms of invertibility conditions for generalized Galois morphisms.

Motivated by the classical theory of Drinfeld twists, we introduce a notion of twist for non-coassociative bimonoids and show that such twists induce deformations of the coproduct yielding new Hopf coquasigroup structures. In particular, we obtain explicit formulas for the corresponding deformed codivisions and antipodes. These constructions provide, to our knowledge, the first systematic treatment of twist deformations in the setting of Hopf coquasigroups.

In \cite[Example 5.11]{KM}, Klim and Majid introduced the Hopf coquasigroup ${\mathbb K}[{\sf S}^7]\rtimes {\mathbb K}G$, where  $G={\mathbb Z}_{2}^{3}$, describing it as "The first example of a quantum Hopf coquasigroup, which we believe to be the first of many". Surprisingly, no further explicit families of quantum Hopf coquasigroups seem to have been constructed in the literature. The present paper addresses  precisely this gap. By applying the developed twist deformation techniques to the Hopf coquasigroup ${\mathbb K}[{\sf S}^7]\rtimes {\mathbb K}G$, we obtain a non-trivial family of quantum Hopf coquasigroups parametrized by the nonzero elements of $G$.

 The paper is organized as follows: In Section~1 we recall the basic categorical notions and establish the notation used throughout the paper. Section~2 is devoted to the study of non-coassociative bimonoids endowed with left and right codivisions. In particular, we analyze their structural properties and clarify their relationship with left and right Hopf coquasigroups. In the first subsection of Section~3 we introduce the notion of a twist for a  non-coassociative bimonoid and develop the corresponding deformation theory for Hopf coquasigroups. More precisely, we prove that suitable twists induce deformations of the coproduct leading to new right and left Hopf coquasigroup structures, and we study the corresponding antipodes. Finally, if ${\sf C }$ is  the category of vector spaces over a field ${\Bbb K}$ and  $G={\mathbb Z}_{2}^{3}$, in the second subsection of Section~3 we construct explicit examples arising from the  algebra ${\mathbb K}[{\sf S}^7]\rtimes {\mathbb K}G$ associated to the sphere ${\sf S}^7$. In this setting, we define a family of involutive twists in ${\mathbb K}[{\sf S}^7]\rtimes {\mathbb K}G$ and compute the corresponding twisted coproducts and antipodes explicitly, obtaining new examples of twisted Hopf coquasigroups that are neither commutative nor cocommutative.

\section{Preliminaries}

Throughout this paper, $\sf C$ denotes a strict  symmetric monoidal category with tensor product $\otimes$, unit object $K$, and natural isomorphism of symmetry $c$. 

Recall that  a monoidal category is a category ${\sf C}$ equipped with a tensor product functor $\otimes :{\sf C}\times {\sf C}\rightarrow {\sf C}$,  a unit object  $K$ of ${\sf C}$ and  a family of natural isomorphisms 
$$a_{M,N,P}:(M\ot N)\ot P\rightarrow M\ot (N\ot P),$$
$$r_{M}:M\ot K\rightarrow M, \;\;\; l_{M}:K\ot M\rightarrow M,$$
in ${\sf C}$ (called  associativity, right unit and left unit constraints, respectively) satisfying the pentagon and the triangle axioms, i.e.,
$$a_{M,N, P\ot Q}\co a_{M\ot N,P,Q}= (id_{M}\ot a_{N,P,Q})\co a_{M,N\ot P,Q}\co (a_{M,N,P}\ot id_{Q}),$$
$$(id_{M}\ot l_{N})\co a_{M,K,N}=r_{M}\ot id_{N},$$
where $id_{X}$ denotes the identity morphism for each object $X$ in ${\sf
C}$. A monoidal category is called strict if the associativity, right unit and left unit constraints are identities. Since every non-strict monoidal category is monoidally equivalent to  a strict one, we may assume without loss of generality that the category is strict  and, as a consequence, the results contained in this paper remain valid for every non-strict symmetric monoidal category. This includes, for instance, ${\mathbb K}$-{\sf Vect}, the category of vector spaces over a field ${\Bbb K}$, or the category of left modules over a commutative ring $R$. For notational simplicity, given objects $M$, $N$, $P$ in ${\sf C}$ and a morphism $f:M\rightarrow N$, we write $P\ot f$ for
$id_{P}\ot f$ and $f \ot P$ for $f\ot id_{P}$.

Recall also that a strict monoidal category  ${\sf C}$ is symmetric  if, for all $M$, $N$ in ${\sf C}$, there exists  a natural isomorphism $c_{M,N}:M\ot N\rightarrow N\ot M$ such that the equalities
$$
c_{M,N\ot P}= (id_{N}\ot c_{M,P})\co (c_{M,N}\ot id_{P}),\;\;
c_{M\ot N, P}= (c_{M,P}\ot id_{N})\co (id_{M}\ot c_{N,P}),\;\; 
$$
$$c_{N,M}\co c_{M,N}=id_{M\ot N},$$
hold for all $P$ in ${\sf C}$.

A magma in ${\sf C}$ is a pair $A=(A,\mu_A)$ consisting of an object $A$ of ${\sf C}$ and a morphism $\mu_A:A\otimes A\to A$, called the product. A unital magma  in ${\sf C}$ is a triple $A=(A, \eta_A, 
\mu_{A})$ where $(A,  
\mu_{A})$ is a magma in ${\sf C}$ and
$\eta_{A}: K\rightarrow A$ is a morphism, called the unit, such that 
$\mu_{A}\circ (A\otimes \eta_{A})=id_{A}=\mu_{A}\circ
(\eta_{A}\otimes A)$. A monoid in ${\sf C}$ is a unital magma $A=(A, \eta_A, \mu_{A})$ whose product is associative, that is,  $\mu_{A}\circ (A\otimes
\mu_{A})=\mu_{A}\circ (\mu_{A}\otimes A)$. Given two unital magmas (monoids) $A$ and $B$,
$f:A\rightarrow B$ is a morphism of unital magmas (monoids) if  $f\circ \eta_{A}= \eta_{B}$ and $\mu_{B}\circ (f\otimes f)=f\circ \mu_{A}$. 

Besides,
if $A$, $B$ are unital magmas (monoids) in ${\sf C}$, the object $A\otimes
B$ is a unital magma (monoid) in
${\sf C}$ where $\eta_{A\otimes B}=\eta_{A}\otimes \eta_{B}$
and $\mu_{A\otimes B}=(\mu_{A}\otimes \mu_{B})\circ (A\otimes c_{B,A}\otimes B)$. 

A comagma in ${\sf C}$ is a pair ${D} = (D, \delta_{D})$ where $D$ is an object in ${\sf C}$ and $\delta_{D}:D\rightarrow D\otimes D$  is a morphism called the coproduct. A counital comagma in ${\sf C}$ is a triple ${D} = (D,
\varepsilon_{D}, \delta_{D})$ where $(D, \delta_{D})$ is a comagma in ${\sf
	C}$ and $\varepsilon_{D}: D\rightarrow K$  is a  morphism, called the counit,  such that $(\varepsilon_{D}\otimes D)\circ
\delta_{D}= id_{D}=(D\otimes \varepsilon_{D})\circ \delta_{D}$.  A comonoid  in ${\sf C}$ is a counital comagma in ${\sf C}$ whose coproduct is coassociative, that is,   $(\delta_{D}\otimes D)\circ \delta_{D}= (D\otimes \delta_{D})\circ \delta_{D}$. If $D$ and
$E$ are counital comagmas (comonoids) in  ${\sf C}$,
$f:D\rightarrow E$ is a  morphism of counital comagmas (comonoids) if $\varepsilon_{E}\circ f
=\varepsilon_{D}$,  and $(f\otimes f)\circ
\delta_{D} =\delta_{E}\circ f$.  

Moreover, if $D$, $E$ are counital comagmas (comonoids) in ${\sf C}$,
the object $D\otimes E$ is a counital comagma (comonoid) in ${\sf C}$ where
$\varepsilon_{D\otimes E}=\varepsilon_{D}\otimes \varepsilon_{E}$
and $\delta_{D\otimes E}=(D\otimes c_{D,E}\otimes E)\circ(
\delta_{D}\otimes  \delta_{E})$.

Let $f:B\rightarrow A$ and $g:B\rightarrow A$ be morphisms between a comagma $B$ and a magma $A$. The convolution product of $f$ and $g$ is defined by $f\ast g=\mu_{A}\circ
(f\otimes g)\circ \delta_{B}$. If $A$ is unital and $B$ is  counital, we say that $f$ is convolution invertible if there exists $f^{-1}:B\to A$ such that $f \ast f^{-1}=f^{-1}\ast f=\varepsilon_B\ot \eta_A$. In the particular case $B=K$, we obtain that $f\ast g=\mu_{A}\circ (f\otimes g)$ and $f$ is convolution invertible if there exists $f^{-1}:K\to A$ such that $f \ast f^{-1}=f^{-1}\ast f= \eta_A$.

\section{Non-coassociative bimonoids}

In this section, we introduce and study non-coassociative bimonoids with a left (right) codivision. In particular, we establish several structural properties and  clarify their relationship with left and right Hopf coquasigroups.

\begin{Definition}
{\rm 
\label{def-nocobi} A non-coassociative bimonoid in the category ${\sf C}$ consists of a monoid $(H, \eta_{H}, \mu_{H})$ and a counital comagma $(H, \varepsilon_{H}, \delta_{H})$ such that $\eta_{H}$ and $\mu_{H}$ are morphisms of counital comagmas (equivalently, $\varepsilon_H$ and $\delta_H$ are morphisms of monoids). Then the following identities are satisfied:
\begin{equation}
	\label{eta-ec}
	\varepsilon_{H}\co \eta_{H}=id_{K},\;\; \delta_{H}\co \eta_{H}=\eta_{H}\ot \eta_{H},
\end{equation}
\begin{equation}
	\label{mu-ec}
	\varepsilon_{H}\co \mu_{H}=\varepsilon_{H}\ot \varepsilon_{H},\;\; \delta_{H}\co \mu_{H}=(\mu_{H}\ot \mu_{H})\co \delta_{H\ot H}.
\end{equation}

A morphism $f:H\rightarrow B$ between non-coassociative bimonoids is a morphism of monoids and counital comagmas.

Observe that replacing the counital comagma structure with a comonoid structure yields the usual notion of bimonoid.
}
\end{Definition}

\begin{Definition}
{\rm 
\label{Rcodivd}
Let $H$ be a non-coassociative bimonoid. Let $\gamma_{R}:H\otimes H\rightarrow H\otimes H$ be the morphism defined by $\gamma_{R}=(H\otimes \mu_{H})\circ (\delta_{H}\otimes H)$. We say that $H$ admits a right codivision if there exists a morphism $r_{H}:H\rightarrow  H\otimes H$ (called a right codivision of $H$) such that 
\begin{equation}
\label{Rcodiv}
\gamma_{R}\circ r_{H}=H\otimes \eta_{H}=(H\otimes \mu_{H})\circ (r_{H}\otimes H)\circ \delta_{H}.
\end{equation}
}
\end{Definition}

The notion of right codivision for a non-coassociative bimonoid admits the corresponding left version defined as follows:

\begin{Definition}
{\rm 
\label{Lcodivd}
Let $H$ be a non-coassociative bimonoid. Let $\gamma_{L}:H\otimes H\rightarrow H\otimes H$ be the morphism defined by $\gamma_{L}=(\mu_{H}\otimes H)\circ (H\otimes \delta_{H})$. We say that $H$ admits a left codivision if there exists a morphism $l_{H}:H\rightarrow H\otimes H$ (called a left codivision of $H$) such that 
\begin{equation}
	\label{Lcodiv}
	\gamma_{L}\circ l_{H}=\eta_{H}\otimes H=(\mu_{H}\otimes H)\circ (H\otimes l_{H})\circ \delta_{H}.
\end{equation}
}
\end{Definition}

In the associative framework of bimonoids, the morphisms $\gamma_{R}$ and $\gamma_{L}$ introduced above coincide with the classical Galois morphisms. It is well known that their invertibility yields the existence of an antipode and, consequently, the bimonoid becomes a Hopf monoid. In the non-coassociative setting considered here, we obtain the following result:

\begin{theorem}
Let $H$ be a non-coassociative bimonoid. The following assertions hold:
\begin{itemize}
\item[(i)]  There exists a right codivision $r_{H}$ for $H$ if and only if $\gamma_{R}$ is an isomorphism. Consequently, $r_{H}$ is uniquely determined as $r_{H}=\gamma_{R}^{-1}\circ (H\otimes \eta_{H})$.
\item[(ii)]  There exists a left  codivision $l_{H}$ for $H$ if and only if $\gamma_{L}$ is an isomorphism. Hence, $l_{H}$ is uniquely determined as $l_{H}=\gamma_{L}^{-1}\circ (\eta_{H}\otimes H)$.
\end{itemize}
\end{theorem}

\begin{proof} We prove only the first statement, since the argument for (ii) is completely analogous.  
	
Assume first that $H$ admits a right codivision $r_H$. Define $\gamma_{R}^{\prime}=(H\otimes \mu_{H})\circ (r_{H}\otimes H)$. We claim that $\gamma_{R}^{\prime}$  is the inverse morphism of $\gamma_{R}$. Indeed, using the associativity of $\mu_{H}$ together with (\ref{Rcodiv}), we obtain
$$\gamma_{R}\circ \gamma_{R}^{\prime}=(H\otimes\mu_{H} )\circ ((\gamma_{R}\circ r_{H})\otimes H)=id_{H\otimes H}$$
and similarly,
$$\gamma_{R}^{\prime}\circ \gamma_{R}=(H\otimes\mu_{H} )\circ (((H\otimes \mu_{H})\circ (r_{H}\otimes H)\circ \delta_{H})\otimes H)=id_{H\otimes H}.$$
Hence $\gamma_R$ is invertible.

Conversely, assume that $\gamma_{R}$ is invertible and define $r_{H}=\gamma_{R}^{-1}\circ (H\otimes \eta_{H})$.  Clearly, $\gamma_{R}\circ r_{H}=H\otimes \eta_{H}$. Thus, it remains to verify the second identity in \eqref{Rcodiv}.

Observe first that associativity of $\mu_H$ implies that 
\begin{equation}
\label{gr1}
\gamma_{R}^{-1}\circ (H\otimes \mu_{H})=(H\otimes \mu_{H})\circ (\gamma_{R}^{-1}\otimes H).
\end{equation}
Therefore, 
\begin{align*}
(H\otimes \mu_{H})\circ (r_{H}\otimes H)\circ \delta_{H}&=(H\otimes \mu_{H})\circ ((\gamma_{R}^{-1}\circ (H\otimes \eta_{H}))\otimes H)\circ \delta_{H}\\
&=\gamma_{R}^{-1}\circ (H\otimes (\mu_{H}\circ (\eta_{H}\otimes H)))\circ \delta_{H}\\
&=\gamma_{R}^{-1}\circ (H\otimes \mu_{H})\circ (\delta_{H}\otimes \eta_{H})\\
&=\gamma_{R}^{-1}\circ\gamma_{R}\circ (H\otimes \eta_{H})\\
&=H\otimes \eta_{H}
\end{align*}
where we have used \eqref{gr1} and the unit identities. Hence, $r_H$ is a right codivision.

Finally, uniqueness follows immediately from \eqref{Rcodiv}, since any right codivision necessarily satisfies
\[
r_H=\gamma_R^{-1}\circ(H\otimes\eta_H).
\]

This completes the proof.
\end{proof}

\begin{Remark}{\rm  Let $H$ be  a non-coassociative bimonoid.
\begin{itemize}
\item[(i)]  If $\gamma_{R}$ is an isomorphism, then a straightforward computation yields
\begin{equation}
\label{gr2}
\gamma_{R}^{-1}\circ\delta_{H}=H\otimes \eta_{H},
\end{equation}
and
\begin{equation}
\label{gr3}
\mu_{H}\circ \gamma_{R}^{-1}=\varepsilon_{H}\otimes H
\end{equation}

Observe also that, by (\ref{gr3}), we obtain 
\begin{equation}
	\label{gr4}
	\mu_{H}\circ r_{H}=\varepsilon_{H}\otimes \eta_{H}.
\end{equation}

Similarly,  if $\gamma_{L}$ is an isomorphism, then the equalities 
$$
\gamma_{L}^{-1}\circ\delta_{H}= \eta_{H}\otimes H,\;\;\;
\mu_{H}\circ \gamma_{L}^{-1}=H\otimes \varepsilon_{H}, \;\;\;\mu_{H}\circ l_{H}=\varepsilon_{H}\otimes \eta_{H}$$
hold.

\item[(ii)]  Assume that  there exists a right codivision $r_{H}$ for $H$. Define $\rho_{H}=(\varepsilon_{H}\otimes H)\circ r_{H}$. Then the identity 
\begin{equation}
\label{gr5}
\rho_{H}\ast id_{H}=\varepsilon_{H}\otimes \eta_{H}
\end{equation}
holds because,  by (\ref{gr1}) and (\ref{gr2}), we obtain the following:
\begin{align*}
\rho_{H}\ast id_{H}&=(\varepsilon_{H}\otimes H)\circ (H\otimes \mu_{H})\circ (\gamma_{R}^{-1}\otimes H)\circ (H\otimes \eta_{H}\otimes H)\circ \delta_{H}\\
&=(\varepsilon_{H}\otimes H)\circ \gamma_{R}^{-1}\circ (H\otimes \mu_{H})\circ (H\otimes \eta_{H}\otimes H)\circ \delta_{H}\\
&=(\varepsilon_{H}\otimes H)\circ \gamma_{R}^{-1}\circ\delta_{H}\\
&=\varepsilon_{H}\otimes \eta_{H}.
\end{align*}

Furthermore,
\begin{equation}
\label{gr6}
(H\otimes \varepsilon_{H})\circ r_{H}=id_{H}
\end{equation}
because 
\begin{align*}
(H\otimes \varepsilon_{H})\circ r_{H}&=(H\otimes \varepsilon_{H})\circ \gamma_{R}^{-1}\circ(H\otimes \eta_{H})\\ &=(H\otimes \varepsilon_{H})\circ\gamma_{R}\circ  \gamma_{R}^{-1}\circ(H\otimes \eta_{H})\\
&=id_{H}.
\end{align*}

Therefore, the identity
\begin{equation}
\label{gr7}
\varepsilon_{H}\circ \rho_{H}=\varepsilon_{H}
\end{equation}
also holds.  Finally, we can prove that 
\begin{equation}
\label{gr8}
\rho_{H}\circ \eta_{H}=\eta_{H}
\end{equation}
because, by (\ref{eta-ec}) and (\ref{gr2}), we have 
$$\rho_{H}\circ \eta_{H}=(\varepsilon_{H}\otimes H)\circ\gamma_{R}^{-1}\circ (\eta_{H}\otimes \eta_{H})
=(\varepsilon_{H}\otimes H)\circ\gamma_{R}^{-1}\circ \delta_{H}\circ \eta_{H}=\eta_{H}.$$

Similarly, if  there exists a left codivision $l_{H}$ for $H$ and  $\lambda_{H}=(H\otimes \varepsilon_{H})\circ l_{H}$,  then the identities  
$$id_{H}\ast \lambda_{H}=\varepsilon_{H}\otimes \eta_{H}, \;\; ( \varepsilon_{H}\otimes H)\circ l_{H}=id_{H}, $$
and 
$$\varepsilon_{H}\circ \lambda_{H}=\varepsilon_{H}, \;\; \lambda_{H}\circ \eta_{H}=\eta_{H}.$$
hold.

\item[(iii)] As in the previous point, suppose that  there exists a right codivision $r_{H}$ for $H$. Then, 
\begin{equation}
\label{gr9}
r_{H}\circ \mu_{H}=(\mu_{H}\otimes (\mu_{H}\circ c_{H,H}))\circ (H\otimes c_{H,H}\otimes H)\circ (r_{H}\otimes r_{H}). 
\end{equation}

Indeed, since $\gamma_{R}$  is invertible, it suffices to show that both morphisms in (\ref{gr9}) are  the same composing with   $\gamma_{R}$. Observe first that, by (\ref{Rcodiv}),
$$\gamma_{R}\circ r_{H}\circ \mu_{H}=\mu_{H}\otimes \eta_{H}$$
holds. On the other hand, we have
\begin{itemize}
	\itemindent=-10pt
	\item[]$\hspace{0.38cm}	\gamma_{R}\circ (\mu_{H}\otimes (\mu_{H}\circ c_{H,H}))\circ (H\otimes c_{H,H}\otimes H)\circ (r_{H}\otimes r_{H})$
	\item[]$= (H\otimes \mu_{H})\circ (((\mu_{H}\ot \mu_{H})\co \delta_{H\ot H})\otimes (\mu_{H}\circ c_{H,H}))\circ (H\otimes c_{H,H}\otimes H)\circ (r_{H}\otimes r_{H})$ {\footnotesize (by  (\ref{mu-ec}))}
	\item[]$=(H\otimes \mu_{H})\circ (\mu_{H\otimes H}\otimes H)\circ (\delta_{H}\otimes (\gamma_{R}\circ r_{H})\otimes H)\circ (H\otimes c_{H,H})\circ (r_{H}\otimes H)$ {\footnotesize (by the associativity of $\mu_{H}$)}
	\item[]$=\mu_{H\otimes H}\circ  (\delta_{H}\otimes  c_{H,H})\circ (r_{H}\otimes H) $ {\footnotesize (by  (\ref{Rcodiv}) and the unit properties)}
	\item[]$=(\mu_{H}\otimes H)\circ (H\otimes c_{H,H})\circ ((\gamma_{R}\circ r_{H})\otimes   H)$ {\footnotesize (by  the naturality of $c$)}
	\item[]$=\mu_{H}\otimes \eta_{H}$ {\footnotesize (by  (\ref{Rcodiv}))}
\end{itemize}

Then, as a consequence, the desired identity  (\ref{gr9}) follows immediately.  Moreover, composing in (\ref{gr9})  with $\varepsilon_{H}\otimes H$, we obtain that 
\begin{equation}
	\label{gr10}
	\rho_{H}\circ \mu_{H}=\mu_{H}\circ c_{H,H}\circ (\rho_{H}\otimes \rho_{H}),
\end{equation}
i.e., $\rho_{H}$ is antimultiplicative. 

Similarly, if there exists a left codivision $l_{H}$ for $H$, then 
$$
l_{H}\circ \mu_{H}=((\mu_{H}\circ c_{H,H})\otimes \mu_{H})\circ (H\otimes c_{H,H}\otimes H)\circ (l_{H}\otimes l_{H})
$$
holds and $\lambda_{H}$ is antimultiplicative, i.e., 
$$
\lambda_{H}\circ \mu_{H}=\mu_{H}\circ c_{H,H}\circ (\lambda_{H}\otimes \lambda_{H}).
$$
\end{itemize}
	
}
\end{Remark}

\begin{Definition}
{\rm 
A right Hopf coquasigroup $H$ is a non-coassociative bimonoid such that there exists a morphism $\rho_{H}:H\rightarrow H$, called the right antipode of $H$, satisfying the equalities:
\begin{equation}
\label{Rhcg}
(H\otimes \mu_{H})\circ (\delta_{H}\otimes \rho_{H})\circ \delta_{H}=H\otimes \eta_{H}=(H\otimes \mu_{H})\circ (((H\otimes \rho_{H})\circ \delta_{H})\otimes H)\circ \delta_{H}.
\end{equation}
}
\end{Definition}

For the left side we have the definition of left Hopf coquasigroup.

\begin{Definition}
{\rm 
A left Hopf coquasigroup $H$ is a non-coassociative bimonoid such that there exists a morphism $\lambda_{H}:H\rightarrow H$, called the left antipode of $H$, satisfying the equalities:
\begin{equation}
\label{Lhcg}
(\mu_{H}\otimes H)\circ (\lambda_{H}\otimes \delta_{H})\circ \delta_{H}=\eta_{H}\otimes H=(\mu_{H}\otimes H)\circ (H\otimes ((\lambda_{H}\otimes H)\circ \delta_{H}))\circ \delta_{H}.
\end{equation}
}
\end{Definition}

\begin{Remark}
{\rm If $H$ is a right Hopf coquasigroup with right antipode $\rho_{H}$, then composing with $\varepsilon_{H}\otimes H$ in (\ref{Rhcg}), we obtain that 
\begin{equation}
\label{Rhcg1}
\rho_{H}\ast id_{H}=id_{H}\ast \rho_{H}=\varepsilon_{H}\otimes \eta_{H},
\end{equation}
or, in other words, $\rho_{H}$ is the convolution inverse of $id_{H}$.

Similarly, if $H$ is a left Hopf coquasigroup with left  antipode $\lambda_{H}$, then composing with $H\otimes \varepsilon_{H}$ in (\ref{Lhcg}) we have that $\lambda_{H}$ is the convolution inverse of $id_{H}$, or equivalently,
 $$\lambda_{H}\ast id_{H}=id_{H}\ast \lambda_{H}=\varepsilon_{H}\otimes \eta_{H}.$$
}
\end{Remark}

Combining the previous definitions of right and left Hopf coquasigroup  leads to the notion of Hopf coquasigroup introduced by J. Klim and S. Majid in \cite{KM}.

\begin{Definition}
{\rm 
A  Hopf coquasigroup $H$ is a non-coassociative bimonoid such that there exists a morphism $S_{H}:H\rightarrow H$, called the antipode of $H$, satisfying the equalities (\ref{Rhcg}) and (\ref{Lhcg}).

A morphism $f:H\rightarrow D$ between  Hopf coquasigroups is a Hopf coquasigroup morphism if it is a morphism of non-coassocitive bimonoids. 
}
\end{Definition}

In this monoidal setting, one can prove, as in \cite{KM}, that  $S_{H}$ is unique and satisfies (\ref{gr7}), (\ref{gr8}), (\ref{gr10}) and 
\begin{equation}
\label{antico}
(S_{H}\otimes S_{H})\circ c_{H,H}\circ \delta_{H}=\delta_{H}\circ S_{H}, 
\end{equation}
i.e., $S_{H}$ leaves the unit and the counit invariant and it is antimultiplicative and anticomultiplicative.

The following example was introduced by J. Klim and S. Majid in \cite{KM} (see also \cite{KM2} and \cite{KThesis}).

\begin{Example}
\label{ex1}
{\rm In this example we will work in a category of vector spaces over a field ${\mathbb K}$. Then we use the term algebra rather than monoid. Consider the abelian group  $G={\mathbb Z}_{2}^{3}$ of ${\mathbb Z}_{2}$-valued vectors and let ${\mathbb K}$ be a field. Denote by ${\mathbb K}G$ the group algebra of $G$ and by $\{\sigma_{u} \;|\; u\in G \}$ its basis. Let $F:G\times G\rightarrow {\mathbb K}$ be the 2-cochain defined by  the following table 
\[
\begin{array}{c|rrrrrrrr}
	F(a,b)
	&000&001&010&011&100&101&110&111\\
	\hline
	000& 1& 1& 1& 1& 1& 1& 1& 1\\
	001& 1&-1& 1&-1& 1&-1& 1&-1\\
	010& 1&-1&-1& 1& 1&-1&-1& 1\\
	011& 1& 1&-1&-1&-1&-1& 1& 1\\
	100& 1&-1&-1& 1&-1& 1& 1&-1\\
	101& 1& 1& 1& 1&-1&-1&-1&-1\\
	110& 1&-1& 1&-1&-1& 1&-1& 1\\
	111& 1& 1&-1&-1& 1& 1&-1&-1
\end{array}
\]

Then, 
$$F(a,b)=(-1)^{\gamma(a,b)}$$
where 
$$\gamma(a,b)=\sum_{1\leq i\leq j\leq 3} a_ib_j+a_1a_2b_3+a_1b_2b_3+b_1a_2b_3.$$

Let ${\mathbb K}_{F} G $ be the vector space with basis $\{e_{a} \;|\; a\in G \}$, i.e. 
${\mathbb K}_{F} G =\displaystyle\bigoplus_{a\in G}{\mathbb K}e_{a}$, and with product defined by 
$$e_{a}\cdot e_{b}=F(a,b)e_{a+b}.$$

As  proved in \cite{AM} (see also \cite[Section 2]{KM}), ${\mathbb K}_{F} G $ with the previous product  is a composition algebra with respect to the non singular quadratic form  on basis $G$ defined by the Euclidean norm,  $N(x)=\displaystyle\sum_{a\in G}x_{a}^2$, for $x=\displaystyle \sum_{a\in G}x_{a}e_{a},$
because 
$$F(a,b)^2=1 $$
holds, for all $a,b\in G$, and
$$F(a,a+c)F(b,b+c)+F(a,b+c)F(b, a+c)=0$$
also holds for all $a\neq b$ and $c$ in $G$.  The statement that  ${\mathbb K}_{F} G $ is a composition algebra means that $N(xy)=N(x)N(y)$, and this implies that the 7-sphere 
over ${\mathbb K}$ is closed under the product defined in ${\mathbb K}_{F} G $. Henceforth, we denote  this sphere as in \cite{KM} by
$${\sf S}^7=\{ x=\sum_{a\in G}x_{a}e_{a}\; |\; \displaystyle\sum_{a\in G}x_{a}^2=1\}.$$

As  proved in \cite[Proposition 3.6]{KM}, ${\sf S}^7$ is an I.P. loop or, in other words, a Hopf quasigroup (see  \cite[Definition 3.1]{KM}) in the category {\sf Set}.

Following \cite[Section 5]{KM}, let ${\mathbb K}[{\sf S}^7]$ be the commutative polynomial algebra ${\mathbb K}[f_{a}\;|\; a\in G]$ with relations $\displaystyle\sum_{a\in G}f_a^2=1.$  Then, ${\mathbb K}[{\sf S}^7]$ is the algebra of functions on the sphere ${\sf S}^7$ with pointwise product generated by the functions $f_{a}(x)=x_{a}$ where $x=\displaystyle \sum_{a\in G}x_{a}e_{a}.$ As was proved in \cite[Proposition 5.7]{KM}, ${\mathbb K}[{\sf S}^7]$ is a Hopf coquasigroup with   counit 
$$\varepsilon_{{\mathbb K}[{\sf S}^7]}(f_{c})=\delta_{c,0},$$
coproduct
$$\delta_{{\mathbb K}[{\sf S}^7]}(f_{c})=\sum_{a+b=c}F(a,b) f_{a}\otimes f_{b},$$
and antipode, 
$$S_{{\mathbb K}[{\sf S}^7]}(f_{c})=F(c,c)f_{c}.$$

The Hopf coquasigroup  ${\mathbb K}[{\sf S}^7]$ has an action of $G$ 
$$\varphi_{{\mathbb K}[{\sf S}^7]}:{\mathbb K}G\otimes {\mathbb K}[{\sf S}^7]\rightarrow {\mathbb K}[{\sf S}^7]$$
defined by 
$$\varphi_{{\mathbb K}[{\sf S}^7]}(\sigma_{u}\otimes 1_{{\mathbb K}[{\sf S}^7]})=1_{{\mathbb K}[{\sf S}^7]}, \;\;\varphi_{{\mathbb K}[{\sf S}^7]}(\sigma_{u}\otimes f_{c})=\chi(c,u)f_{c}$$
where $$\chi(c,u)=(-1)^{uc}$$ and  $uc$ denotes the ${\mathbb Z}_{2}$-scalar product. 

By \cite[Proposition 5.10, Example 5.11]{KM}, there exists a cross product $$D={\mathbb K}[{\sf S}^7]\rtimes {\mathbb K}G$$ such that it is a  non-commutative and non-cocommutative Hopf coquasigroup with the following structure: Put $$T_{1}^u=1_{{\mathbb K}[{\sf S}^7]}\otimes \sigma_{u}, \;\; T_{c}^u=f_{c}\otimes \sigma_{u}, \;\;
T_{c, d}^u=f_{c}f_d\otimes \sigma_{u}$$
for $c,d, u \in G$.  Then, the unit is $\eta_{D}:{\mathbb K}\rightarrow D$ defined by 
\begin{equation}
	\label{etaD}
	\eta_{D}(1)=1_{{\mathbb K}[{\sf S}^7]}\otimes \sigma_{0}=T_{1}^0=1_{D},
\end{equation}
the product is defined by the morphism $\mu_{D}:D\otimes D\rightarrow D$ where, if we denote $\mu_{D}(M\otimes N)$ by $M\bullet N$, 
\begin{equation}
	\label{muD1}
	T_{1}^u\bullet T_{1}^{v}=T_{1}^{u+v},
\end{equation}
\begin{equation}
	\label{muD2}
	T_{1}^u\bullet T_{c}^{v}=\varphi_{{\mathbb K}[{\sf S}^7]}(\sigma_{u}\otimes f_{c})\otimes  \sigma_{u+v}=\chi(c,u)T_{c}^{u+v},
\end{equation}
\begin{equation}
	\label{muD3}
	T_{c}^u\bullet T_{1}^{v}=f_{c}\varphi_{{\mathbb K}[{\sf S}^7]}(\sigma_{u}\otimes 1_{{\mathbb K}[{\sf S}^7]})\otimes  \sigma_{u+v}=T_{c}^{u+v},
\end{equation}
\begin{equation}
	\label{muD4}
	T_{c}^u\bullet T_{b}^v=f_{c}\varphi_{{\mathbb K}[{\sf S}^7]}(\sigma_{u}\otimes f_{b})\otimes  \sigma_{u+u'}= \chi(b,u)T_{c,b}^{u+v}.
\end{equation}

The counit is the morphism $\varepsilon_{D}: D\rightarrow {\mathbb K}$ with 
\begin{equation}
	\label{varepD}
	\varepsilon_{D}(T_{1}^u)=1, \;\; \varepsilon_{D}(T_{c}^u)=\delta_{0,c},
\end{equation}
and the coproduct is the morphism $\delta_{D}: D\rightarrow D\otimes D$ such that 
\begin{equation}
	\label{copD1}
	\delta_{D}(T_{1}^u)=T_{1}^u\otimes  T_{1}^u,
\end{equation}
\begin{equation}
	\label{copD2}
	\delta_{D}(T_{c}^u)=\sum_{a+b=c}F(a,b) T_{a}^u\otimes T_{b}^u.
\end{equation}

Finally, the antipode  is the morphism $S_{D}:D\rightarrow D$ satisfying
\begin{equation}
	\label{antD1}
	S_{D}(T_{1}^u)=T_{1}^u
\end{equation}
and 
\begin{equation}
	\label{antD2}
	S_{D}(T_{c}^u)= \varphi_{{\mathbb K}[{\sf S}^7]}(\sigma_{u}\otimes S_{{\mathbb K}[{\sf S}^7]}(f_{c}))\otimes \sigma_u=\chi(c,u)F(c,c)T_{c}^u.
\end{equation}
}
\end{Example}

We conclude this section by establishing a new characterization of right  (left) Hopf coquasigroups.

\begin{theorem}
\label{eqivRTH}
The following assertions are equivalent for a non-coassociative bimonoid $H$.
\begin{itemize}
\item[(i)]  There exists a right codivision $r_{H}$ for $H$ such that 
\begin{equation}
\label{gRf1} 
r_{H}=(H\otimes \rho_{H})\circ \delta_{H}
\end{equation}
holds for $\rho_{H}=(\varepsilon_{H}\otimes H)\circ r_{H}.$
\item[(ii)] There exists a right codivision $r_{H}$ for $H$ such that 
\begin{equation}
	\label{gRf2} 
	(H\otimes \mu_{H})\circ (\delta_{H}\otimes \rho_{H})\circ \delta_{H}=H\otimes \eta_{H}
\end{equation}
holds for $\rho_{H}=(\varepsilon_{H}\otimes H)\circ r_{H}.$
\item[(iii)]  The non-coassociative bimonoid $H$ is a right Hopf coquasigroup.
\end{itemize}
\end{theorem}
\begin{proof}
We prove the equivalence by showing (i) $\Rightarrow$ (ii) $\Rightarrow$ (i) $\Rightarrow$ (iii) $\Rightarrow$ (i).

Assume first that (i) holds. Since
	\[
	r_H=(H\otimes\rho_H)\circ\delta_H,
	\]
identity \eqref{gRf2} follows immediately from \eqref{Rcodiv}, because
	\[
	(H\otimes\mu_H)\circ(\delta_H\otimes\rho_H)\circ\delta_H
	=
	\gamma_R\circ r_H
	=
	H\otimes\eta_H.
	\]
Thus (ii) holds.
	
Conversely, assume (ii). Let us show that
	\[
	r_H=(H\otimes\rho_H)\circ\delta_H.
	\]
Indeed, using the unit identities, associativity of $\mu_H$, and \eqref{Rcodiv}, we compute
	\begin{align*}
		r_H
		&=
		(H\otimes\mu_H)\circ(r_H\otimes\eta_H)
		\\
		&=
		(H\otimes\mu_H)
		\circ
		(r_H\otimes\mu_H)
		\circ
		(\delta_H\otimes\rho_H)
		\circ
		\delta_H
		\\
		&=
		(H\otimes\mu_H)
		\circ
		(((H\otimes\mu_H)\circ(r_H\otimes H)\circ\delta_H)
		\otimes\rho_H)
		\circ
		\delta_H
		\\
		&=
		(H\otimes\rho_H)\circ\delta_H,
	\end{align*}
which proves (i).
	
Next assume (i). By the equivalence already established between (i) and (ii), the first identity in \eqref{Rhcg} holds automatically. Moreover, using \eqref{gRf1} together with \eqref{Rcodiv}, we obtain
	\[
	(H\otimes\mu_H)
	\circ
	(((H\otimes\rho_H)\circ\delta_H)\otimes H)
	\circ
	\delta_H
	=
	(H\otimes\mu_H)
	\circ
	(r_H\otimes H)
	\circ
	\delta_H
	=
	H\otimes\eta_H.
	\]
Hence the second identity in \eqref{Rhcg} follows as well, and therefore $H$ is a right Hopf coquasigroup. Thus (iii) holds.
	
Finally, assume  (iii), that is, $H$ is a right Hopf coquasigroup with right antipode $\rho_H$, and define
	\[
	r_H=(H\otimes\rho_H)\circ\delta_H.
	\]
Then, using \eqref{Rhcg}, we obtain that 
	\[
	\gamma_R\circ r_H
	=
	(H\otimes\mu_H)
	\circ
	(\delta_H\otimes\rho_H)
	\circ
	\delta_H
	=
	H\otimes\eta_H
	\]
and
	\[
	(H\otimes\mu_H)
	\circ
	(r_H\otimes H)
	\circ
	\delta_H
	=
	(H\otimes\mu_H)
	\circ
	(((H\otimes\rho_H)\circ\delta_H)\otimes H)
	\circ
	\delta_H
	=
	H\otimes\eta_H.
	\]
Therefore $r_H$ satisfies the defining identities of a right codivision, proving (i).
	
This completes the proof.
\end{proof}

For the left case we have a similar result that asserts the following.

\begin{theorem}
The following assertions are equivalent for a non-coassociative bimonoid $H$.
\begin{itemize}
\item[(i)]  There exists a left codivision $l_{H}$ for $H$ such that 
$$
l_{H}=(\lambda_{H}\otimes H)\circ \delta_{H}
$$
holds for $\lambda_{H}=(H\otimes \varepsilon_{H})\circ l_{H}.$
\item[(ii)] There exists a left  codivision $l_{H}$ for $H$ such that 
$$
(\mu_{H}\otimes H)\circ (\lambda_{H}\otimes \delta_{H})\circ \delta_{H}=\eta_{H}\otimes H
$$
holds for $\lambda_{H}=(H\otimes \varepsilon_{H})\circ l_{H}.$
\item[(iii)]  The non-coassociative bimonoid $H$ is a left Hopf coquasigroup.
	\end{itemize}
\end{theorem}

\section{{\sc Twist deformations}}

In this section we will present the general theory of twist deformations for Hopf coquasigroups and we apply it to the Hopf coquasigroup $D={\mathbb K}[{\sf S}^7]\rtimes {\mathbb K}G$ introduced in Example \ref{ex1}.

\subsection{{\sc General theory}} In the theory of Hopf algebras, new Hopf algebras can be obtained by twisting either the algebra or the coalgebra structure. As  proved in \cite{NMR}, in the non associative setting it is possible to modify the product of a Hopf quasigroup using a 2-cocycle  to obtain a new Hopf quasigroup. Since the notion of a Hopf coquasigroup is dual to that of a Hopf quasigroup, in this subsection we develop a general theory of twist deformations for Hopf coquasigroups by deforming their coproducts via twists of non-coassociative bimonoids. These twists correspond to the dual analogues of the 2-cocycles introduced in  \cite{NMR}.

\begin{Definition}
{\rm 
\label{twist} Let $H$ be a non-coassociative bimonoid. Let $J:K\rightarrow H\otimes H$ be a convolution invertible 
morphism. We will say that $J$ is a twist if 
\begin{equation}
\label{twist-def}
\partial_{3}(J)\ast \partial_{1}(J)=\partial_{2}(J)\ast \partial_{4}(J),
\end{equation}
where $\partial_{i}(J):K\rightarrow  H\otimes H\otimes H$, $i\in\{1,2,3,4\}$, are morphisms defined by 
$$ \partial_{1}(J)=J\otimes \eta_{H}, \;\;  \partial_{2}(J)=(H\otimes \delta_{H})\circ J,  \;\;  \partial_{3}(J)=(\delta_{H}\otimes H)\circ J, \;\; \partial_{4}(J)=\eta_{H}\otimes J.$$
}
\end{Definition}

A straightforward computation yields that, for all $i\in\{1,2,3,4\}$, $\partial_{i}(J)$ is convolution invertible and $\partial_{i}(J)^{-1}=\partial_{i}(J^{-1})$. Then, we can obtain the following equivalent conditions to \eqref{twist-def}:  
\begin{equation}
	\label{twist-def1}
	\partial_{1}(J)\ast \partial_{4}(J^{-1})=\partial_{3}(J^{-1})\ast \partial_{2}(J),
\end{equation}
\begin{equation}
	\label{twist-def2}
	\partial_{4}(J^{-1})\ast \partial_{2}(J^{-1})=\partial_{1}(J^{-1})\ast \partial_{3}(J^{-1}), 
\end{equation}
\begin{equation}
\label{twist-def3}
\partial_{4}(J)\ast \partial_{1}(J^{-1})=\partial_{2}(J^{-1})\ast \partial_{3}(J)  
\end{equation}

In the previous identities, if we compute the convolution product,  \eqref{twist-def} is equivalent to
\begin{equation}
	\label{twist-def1n}
((\mu_{H\otimes H}\circ (\delta_{H}\otimes J))\otimes H)\circ J=(H\otimes ((\mu_{H\otimes H}\circ (\delta_{H}\otimes J))))\circ J, 
\end{equation}
 \eqref{twist-def1} is equivalent to
\begin{equation}
	\label{twist-def4n}
	(H\otimes \mu_{H}\otimes H)\circ (J\otimes J^{-1})=\mu_{H\otimes H\otimes H}\circ (((\delta_{H}\otimes H)\circ J^{-1})\otimes (( H\otimes \delta_{H})\circ J)),
\end{equation}
 \eqref{twist-def2} is equivalent to
 \begin{equation}
 	\label{twist-def2n}
 (H\otimes (\mu_{H\otimes H}\circ (J^{-1}\otimes \delta_{H})) )\circ J^{-1}=((\mu_{H\otimes H}\circ (J^{-1}\otimes \delta_{H}))\otimes H)\circ J^{-1}, 
 \end{equation}
 and, finally, \eqref{twist-def3} is equivalent to
\begin{equation}
\label{twist-def3n}
(H\otimes (\mu_{H}\circ c_{H,H})\otimes H)\circ (J^{-1}\otimes J)=\mu_{H\otimes H\otimes H}\circ (((H\otimes \delta_{H})\circ J^{-1})\otimes (( \delta_{H}\otimes H)\circ J)).
\end{equation}

\begin{Remark}
{\rm If ${\sf C}$ is the category of vector spaces over a field ${\mathbb K}$ the notion of twist used in this paper is the one of right twist that we can find in \cite{RAD} for bialgebras, i.e., the original notion introduced by V. Drinfeld in \cite{VD}. For this reason, this type of twist is also known in the literature as a Drinfeld twist.}
\end{Remark}

\begin{Lemma}
Let $H$ be a non-coassociative bimonoid. Let $J:K\rightarrow H\otimes H$ be a convolution invertible morphism. Then, the identity 
\begin{equation}
\label{extra1}
\mu_{H\otimes H}\circ (\mu_{H\otimes H}\otimes H\otimes H)\circ ((H\otimes\mu_{H}\otimes  J\otimes J^{-1}))=H\otimes \mu_{H}, 
\end{equation}
holds.
\end{Lemma}

\begin{proof}
Note that by the associativity of $\mu_{H\otimes H}$, the convolution invertible condition for $J$, and the unit properties, we have that 
$$\mu_{H\otimes H}\circ (\mu_{H\otimes H}\otimes H\otimes H)\circ ((H\otimes \mu_{H}\otimes J\otimes J^{-1}))=\mu_{H\otimes H}\circ (H\otimes \mu_{H}\otimes (J\ast J^{-1}))=H\otimes \mu_{H}$$
and we obtain \eqref{extra1}.

\end{proof}

\begin{Definition}
{\rm
\label{conormal} Let $H$ be a non-coassociative bimonoid. Let $J:K\rightarrow H\otimes H$ be a convolution invertible morphism. We will say that $J$ is conormal if 
\begin{equation}
\label{conor-defn}
(\varepsilon_{H}\otimes H)\circ J=\eta_{H}=(H\otimes \varepsilon_{H})\circ J
\end{equation}
holds.
}
\end{Definition}

A routine computation shows that, if $J$ is  conormal, then $J^{-1}$ also is. Furthermore, if $J$ is a twist, then $\overline{J}=J\otimes ((\varepsilon_{H}\otimes \varepsilon_{H})\circ J^{-1})$ is a conormal twist with convolution inverse $\overline{J}^{-1}=J^{-1}\otimes ((\varepsilon_{H}\otimes \varepsilon_{H})\circ J)$. Consequently, in what follows we may assume, without loss of generality, that $J$ is a conormal twist.

\begin{theorem}
\label{firstmain}
Let $H$ be a non-coassociative bimonoid. Let $J$ be a twist. Define the coproduct 
$$\delta_{H_{J}}=\mu_{H\otimes H}\circ (\mu_{H\otimes H}\otimes H\otimes H)\circ (J^{-1}\otimes \delta_{H}\otimes J).$$

Then,
$$H_{J}=(H, \eta_{H_{J}}=\eta_{H}, \mu_{H_{J}}=\mu_{H}, \varepsilon_{H_{J}}=\varepsilon_{H}, \delta_{H_{J}})$$
 a non-coassociative bimonoid called the twist deformation of $H$ by $J$.
\end{theorem}

\begin{proof}
Since $H_{J}$ has the same monoid structure as $H$ and the two first identities of \eqref{eta-ec} and \eqref{mu-ec} hold, to establish that $H_{J}$ is a non-coassociative bimonoid, it is suffices to show that 
$(H,\varepsilon_{H}, \delta_{H_{J}})$ is a counital comagma and the two last equalities of \eqref{eta-ec} and \eqref{mu-ec} hold.  Indeed, since $\varepsilon_H$ is a morphism of unital magmas and $J$ is conormal, using the counit properties, we have that 
$$(\varepsilon_{H}\otimes H)\circ \delta_{H_{J}}=\mu_{H}\circ ((\mu_{H}\circ (((\varepsilon_{H}\otimes H)\circ J^{-1})\otimes H)\otimes ((\varepsilon_{H}\otimes H)\circ J)))=id_{H},$$
and 
$$(H\otimes \varepsilon_{H})\circ \delta_{H_{J}}=\mu_{H}\circ ((\mu_{H}\circ (((H\otimes \varepsilon_{H})\circ J^{-1})\otimes H)\otimes ((H\otimes \varepsilon_{H})\circ J)))=id_{H}.$$
Thus, $(H,\varepsilon_{H}, \delta_{H_{J}})$ is a counital comagma. 

On the other hand, using that $\delta_H$ is a morphism of unital magmas and the unit properties, we obtain that 
$$\delta_{H_{J}}\circ \eta_{H}=J^{-1}\ast J=\eta_{H}\otimes \eta_{H}.$$

Finally, by the associativity of $\mu_{H}$, and the condition of morphism of unital magmas for $\delta_{H}$,  it follows that 
\begin{align*}
&\;\;\;\;\mu_{H\otimes H}\circ (\delta_{H_{J}}\otimes \delta_{H_{J}})\\
	&=
\mu_{H\otimes H}\circ ((\mu_{H\otimes H}\circ (J^{-1}\otimes \delta_{H}))\otimes (\mu_{H\otimes H}\circ ((J\ast J^{-1})\otimes (\mu_{H\otimes H}\circ (\delta_{H}\otimes J)))))
	\\
	&=
\mu_{H\otimes H}\circ (\mu_{H\otimes H}\otimes H\otimes H)\circ (J^{-1}\otimes (\mu_{H\otimes H}\circ (\delta_{H}\otimes \delta_{H}))\otimes J)
	\\
	&=
\delta_{H_{J}}\circ \mu_{H}.
\end{align*}

Therefore, $H_{J}$ is a non-coassociative bimonoid. 
\end{proof}

\begin{Remark}
{\rm 
Using the associativity of $\mu_{H\otimes H}$, the condition of convolution invertible for $J$, the naturality of $c$,  and the unit properties, for the coproduct  $\delta_{H_{J}}$ we obtain the following identities:
\begin{equation}
\label{ofin1}
\mu_{H\otimes H}\circ (J\otimes \delta_{H_{J}})=\mu_{H\otimes H}\circ (\delta_{H}\otimes J)
\end{equation}
and 
\begin{equation}
\label{ofin2}
\mu_{H\otimes H}\circ ( \delta_{H_{J}}\otimes J^{-1})=\mu_{H\otimes H}\circ (J^{-1}\otimes \delta_{H}).
\end{equation}
}
\end{Remark}

\begin{Remark}
{\rm 
Note that if $H$ is commutative, then for all twist $J$, $H=H_{J}$.}
\end{Remark}

\begin{Lemma}
\label{VR}
Let $H$ be a non-coassociative bimonoid with right codivision $r_{H}$, and let $J$ be a twist. Put $\rho_{H}=(\varepsilon_{H}\otimes H)\circ r_{H}$. Define the morphism
$$v_{R}=\mu_{H}\circ (\rho_{H}\otimes H)\circ J:K\rightarrow H.$$

If $id_{H}\ast \rho_{H}=\varepsilon_{H}\otimes \eta_{H}$ holds, then 
$v_{R}$ is convolution invertible with inverse $$v_{R}^{-1}=\mu_{H}\circ (H\otimes \rho_{H})\circ J^{-1}$$ and satisfies that 
\begin{equation}
\label{VR1}
\varepsilon_{H}\circ v_{R}=\varepsilon_{H}\circ v_{R}^{-1}=id_{K}, 
\end{equation}
\begin{equation}
	\label{VR1-new}
\mu_{H}\circ (H\otimes \rho_{H})\circ \mu_{H\otimes H}\circ (J^{-1}\otimes \delta_{H})=v_{R}^{-1}\otimes \varepsilon_{H}.
\end{equation}
\end{Lemma}

\begin{proof}
We first determine $v_{R}\ast v_{R}^{-1}$: 

\begin{itemize}
\itemindent=-22pt
\item[]$\hspace{0.38cm}	v_{R}\ast v_{R}^{-1}$

\item[]$= \mu_{H}\circ ((\mu_{H}\circ (\rho_{H}\otimes H))\otimes \rho_{H})\circ (H\otimes \mu_{H}\otimes H)\circ (J\otimes J^{-1})$ {\footnotesize (by   the associativity of $\mu_{H}$)}

\item[]$= \mu_{H}\circ ((\mu_{H}\circ (\rho_{H}\otimes H))\otimes \rho_{H})\circ (\partial_{1}(J)\ast \partial_{4}(J^{-1}))$ {\footnotesize (by   the convolution product)}

\item[]$= \mu_{H}\circ ((\mu_{H}\circ (\rho_{H}\otimes H))\otimes \rho_{H})\circ (\partial_{3}(J^{-1})\ast \partial_{2}(J))$ {\footnotesize (by   \eqref{twist-def1})}

\item[]$= \mu_{H}\circ ((\mu_{H}\circ ((\rho_{H}\circ \mu_{H})\otimes \mu_{H}))\otimes (\rho_{H}\circ \mu_{H}))\circ (H\otimes c_{H,H}\otimes c_{H,H}\otimes H)\circ (\delta_{H}\otimes c_{H,H}\otimes \delta_{H}) \circ  (J^{-1}\otimes J)$ 
\item[]$\hspace{0.38cm}${\footnotesize (by  the convolution product)}

\item[]$=\mu_{H}\circ (\mu_{H}\otimes H)\circ (H\otimes \mu_{H}\otimes H)
\circ (((\rho_{H}\otimes \rho_{H})\circ c_{H,H})\otimes (\mu_{H}\circ (H\otimes (id_{H}\ast \rho_{H})))\otimes \rho_{H})$
\item[]$\hspace{0.38cm}	\circ (H\otimes c_{H,H}\otimes c_{H,H}) \circ (\delta_{H}\otimes c_{H,H}\otimes H)\circ (J^{-1}\otimes J)$ 
{\footnotesize (by  \eqref{gr10}, the associativity of $\mu_{H}$, and the naturality}
\item[]$\hspace{0.38cm}${\footnotesize of $c$)}

\item[]$=\mu_{H}\circ (\rho_{H}\otimes \mu_{H})\circ (\delta_{H}\otimes \rho_{H})\circ J^{-1}$
{\footnotesize (by $id_{H}\ast \rho_{H}=\varepsilon_{H}\otimes \eta_{H}$, the conormal condition for $J$, the naturality}
\item[]$\hspace{0.38cm}${\footnotesize  of $c$, \eqref{gr8}, and the unit properties)}

\item[]$=(\varepsilon_{H}\otimes \mu_{H})\circ (((H\otimes \mu_{H})\circ (r_{H}\otimes H)\circ \delta_{H})\otimes \rho_{H})\circ J^{-1}$
{\footnotesize (by  associativity of $\mu_{H}$)}

\item[]$=(\varepsilon_{H}\otimes (\mu_{H}\circ (\eta_{H}\otimes \rho_{H}))\circ  J^{-1}$
{\footnotesize (by  \eqref{Rcodiv})}

\item[]$=\rho_{H}\circ \eta_{H}$
{\footnotesize (by  the conormal condition for $J^{-1}$ and the unit properties)}

\item[]$= \eta_{H}$
{\footnotesize (by  \eqref{gr8}).}

\end{itemize}

Next, we compute $v_{R}^{-1}\ast v_{R}$: 

\begin{itemize}
	\itemindent=-22pt
	\item[]$\hspace{0.38cm}	v_{R}^{-1}\ast v_{R}$
	
	\item[]$= \mu_{H}\circ (\mu_{H}\otimes H)\circ (H\otimes (\mu_{H}\circ (\rho_{H}\otimes \rho_{H}))\otimes H) \circ (J^{-1}\otimes J)$ {\footnotesize (by  the associativity of $\mu_{H}$)}
	
	\item[]$= \mu_{H}\circ (\mu_{H}\otimes H)\circ (H\otimes (\rho_{H}\circ \mu_{H}\circ c_{H,H})\otimes H) \circ (J^{-1}\otimes J)$ {\footnotesize (by  \eqref{gr10} and the symmetric condition of}
	\item[]$\hspace{0.38cm}	${\footnotesize  the category ${\sf C}$)}
	
	\item[]$= \mu_{H}\circ ((\mu_{H}\circ (H\otimes \rho_{H}))\otimes H)\circ (\partial_{4}(J)\ast \partial_{1}(J^{-1}))$ {\footnotesize (by   the convolution product)}
	
	\item[]$= \mu_{H}\circ ((\mu_{H}\circ (H\otimes \rho_{H}))\otimes H)\circ (\partial_{2}(J^{-1})\ast \partial_3(J))$ {\footnotesize (by   \eqref{twist-def3})}
	
	\item[]$= \mu_{H}\circ (\mu_{H}\otimes H)\circ(\mu_{H}\otimes (\rho_{H}\circ \mu_{H})\otimes \mu_{H})\circ (H\otimes ((c_{H,H}\otimes c_{H,H})\circ \delta_{H\otimes H})\otimes H)\circ (J^{-1}\otimes J)$
	\item[]$\hspace{0.38cm}${\footnotesize (by   the convolution product)}
	
	\item[]$= \mu_{H}\circ (\mu_{H}\otimes \mu_{H})\circ 
	(H\otimes (((id_{H}\ast \rho_{H})\otimes (\mu_{H}\circ (\rho_{H}\otimes H)))\circ ( c_{H,H}\otimes H)\circ (H\otimes c_{H,H})\circ (\delta_{H}\otimes H))\otimes H)$
	\item[]$\hspace{0.38cm}\circ (J^{-1}\otimes J)$ {\footnotesize (by \eqref{gr10}, the associativity of $\mu_{H}$, and the naturality of $c$)}
	
	\item[]$= \mu_{H}\circ (H\otimes ((\varepsilon_{H}\otimes \mu_{H})\circ (r_{H}\otimes H)\circ \delta_{H}))\circ J^{-1}$ 
	{\footnotesize (by $id_{H}\ast \rho_{H}=\varepsilon_{H}\otimes \eta_{H}$, the conormal condition for}
	\item[]$\hspace{0.38cm}${\footnotesize    $J$, the naturality of $c$, and the unit properties)}
	
	\item[]$= \eta_{H}$
	{\footnotesize (by  \eqref{Rcodiv}, the conormal condition for $J^{-1}$ and the unit properties).}
	
\end{itemize}

Therefore, $v_{R}$ is convolution invertible, with inverse given by $v_{R}^{-1}=\mu_{H}\circ (H\otimes \rho_{H})\circ J^{-1}$. 

The equality, \eqref{VR1} follows by \eqref{mu-ec},  \eqref{gr7}, and the conormal condition for $J$ and $J^{-1}$.  Finally, the proof of \eqref{VR1-new} is 
\begin{align*}
&\;\;\;\; \mu_{H}\circ (H\otimes \rho_{H})\circ \mu_{H\otimes H}\circ (J^{-1}\otimes \delta_{H})\\
& =\mu_{H}\circ ((\mu_{H}\circ (H\otimes (id_{H}\ast \rho_{H})))\otimes H)\circ (H\otimes c_{H,H})\circ (((H\otimes \rho_{H})\circ J^{-1})\otimes H) \\
&= v_{R}^{-1}\otimes \varepsilon_{H}, 
\end{align*}
where the first equality follows by \eqref{gr10}, the naturality of $c$, and the associativity of $\mu_{H}$. The second identity follows by $id_{H}\ast \rho_{H}=\varepsilon_{H}\otimes \eta_{H}$, the unit properties and the naturality of $c$.

\end{proof}

The left version of Lemma \ref{VR} admits a similar proof and is the following:

\begin{Lemma}
	\label{VL}
	Let $H$ be a non-coassociative bimonoid with left codivision $l_{H}$ and let $J$ be a twist. Put $\lambda_{H}=(H\otimes \varepsilon_{H})\circ l_{H}$. Define the morphism
	$$v_{L}=\mu_{H}\circ (H\otimes \lambda_{H})\circ J^{-1}:K\rightarrow H.$$
	
	If $\lambda_{H}\ast id_{H}= \varepsilon_{H}\otimes \eta_{H}$ holds, then 
	$v_{L}$ is convolution invertible with inverse $$v_{L}^{-1}=\mu_{H}\circ (\lambda_{H}\otimes H)\circ J$$ and satisfies that
	\begin{equation*}
		\varepsilon_{H}\circ v_{L}=\varepsilon_{H}\circ v_{L}^{-1}=id_{K},
	\end{equation*}
\begin{equation*}
	\mu_{H}\circ (\lambda_{H}\otimes H)\circ \mu_{H\otimes H}\circ (\delta_{H}\otimes J)= \varepsilon_{H}\otimes v_{L}^{-1}.
\end{equation*}
\end{Lemma}

\begin{Lemma}
\label{VRR}
Let $H$ be a non-coassociative bimonoid with right codivision $r_{H}$, and let $J$ be a twist. Put $\rho_{H}=(\varepsilon_{H}\otimes H)\circ r_{H}$ and assume that  $id_{H}\ast \rho_{H}=\varepsilon_{H}\otimes \eta_{H}$. Let $v_{R} $ and $v_{R}^{-1} $ be the morphisms defined in Lemma \ref{VR}. Define the morphism $r_{H_{J}}: H\rightarrow H\otimes H$ by 
$$r_{H_{J}}=(H\otimes (\mu_{H}\circ ((\mu_{H}\circ (v_{R}^{-1}\otimes \rho_{H}))\otimes v_{R})))\circ \delta_{H_{J}},$$
where $\delta_{H_{J}}$ is the coproduct defined in Theorem \ref{firstmain}.

Then,  
\begin{equation}
\label{VRR1}
r_{H_{J}}\circ \mu_{H}=(\mu_{H}\otimes (\mu_{H}\circ c_{H,H}))\circ (H\otimes c_{H,H}\otimes H)\circ (r_{H_{J}}\otimes r_{H_{J}})
\end{equation}
holds. Moreover, 
\begin{equation}
	\label{VRR2}
(H\otimes \varepsilon_{H})\circ	r_{H_{J}}=id_{H}, 
\end{equation}
and 
\begin{equation}
	\label{VRR3}
(\varepsilon_{H}\otimes H)\circ	r_{H_{J}}=\mu_{H}\circ ((\mu_{H}\circ (v_{R}^{-1}\otimes \rho_{H})\otimes v_{R}))
\end{equation}
also holds. Consequently, we have the following equalities:
\begin{equation}
\label{VRR4}
(H\otimes ((\varepsilon_{H}\otimes H)\circ	r_{H_{J}}))\circ \delta_{H_{J}} =r_{H_{J}}
\end{equation}
and 
\begin{equation}
	\label{VRR5}
	(H\otimes \mu_{H})\circ	(r_{H_{J}}\otimes v_{R}^{-1})=(H\otimes (\mu_{H}\circ (v_{R}^{-1}\otimes \rho_{H})))\circ \delta_{H_{J}}.
\end{equation}

Finally, if $H$ is commutative, $\rho_{H}=(\varepsilon_{H}\otimes H)\circ r_{H_{J}}$.

\end{Lemma}

\begin{proof} The proof for \eqref{VRR1} is the following:
	
\begin{itemize}
	\itemindent=-22pt
	\item[]$\hspace{0.38cm}	(\mu_{H}\otimes (\mu_{H}\circ c_{H,H}))\circ (H\otimes c_{H,H}\otimes H)\circ (r_{H_{J}}\otimes r_{H_{J}})$
	
	\item[]$=(\mu_{H}\otimes (\mu_{H}\circ ((\mu_{H}\circ ((\mu_{H}\circ (v_{R}^{-1}\otimes \rho_{H}))\otimes v_{R}))\otimes (\mu_{H}\circ ((\mu_{H}\circ (v_{R}^{-1}\otimes \rho_{H}))\otimes v_{R})))  \circ  c_{H,H}))$
	\item[]$\hspace{0.38cm}\circ (H\otimes c_{H,H}\otimes H)\circ (\delta_{H_{J}}\otimes \delta_{H_{J}}) $ {\footnotesize (by  the naturality of $c$)}
	
	\item[]$=(\mu_{H}\otimes (\mu_{H}\circ (\mu_{H}\otimes H)\circ (H\otimes \mu_{H}\otimes H)\circ (v_{R}^{-1}\otimes (\mu_{H}\circ (\rho_{H}\otimes (v_{R}\ast v_{R}^{-1})))\otimes \rho_{H} \otimes v_{R})\circ  c_{H,H})) $
     \item[]$\hspace{0.38cm}\circ (H\otimes c_{H,H}\otimes H)\circ (\delta_{H_{J}}\otimes \delta_{H_{J}}) $ {\footnotesize (by   the associativity  of $\mu_{H}$)}
	
	\item[]$=(\mu_{H}\otimes (\mu_{H}\circ (\mu_{H}\otimes H)\circ (v_{R}^{-1}\otimes (\mu_{H}\circ (\rho_{H}\otimes \rho_{H} )\circ  c_{H,H})\otimes v_{R})))\circ (H\otimes c_{H,H}\otimes H)\circ (\delta_{H_{J}}\otimes \delta_{H_{J}})$
	\item[]$\hspace{0.38cm} $ {\footnotesize (by  the convolution invertibility of $v_{R}$ and unit properties)}

   \item[]$=(\mu_{H}\otimes (\mu_{H}\circ (\mu_{H}\otimes H)\circ (v_{R}^{-1}\otimes (\rho_{H}\circ \mu_{H})\otimes v_{R})))\circ (H\otimes c_{H,H}\otimes H)\circ (\delta_{H_{J}}\otimes \delta_{H_{J}})$ {\footnotesize (by  \eqref{gr10})}
	
    \item[]$=r_{H_{J}}\circ \mu_{H}$
	{\footnotesize (by the condition of non-coassociative bimonoid of $H_{J}$ (see Theorem \ref{firstmain})).}
	
\end{itemize}

On the other hand, equality \eqref{VRR2} follows using \eqref{mu-ec}, \eqref{gr7}, the condition of counit of $\varepsilon_{H}$ for $H_{J}$, and \eqref{VR1}. Applying the condition of counit of $\varepsilon_{H}$ for $H_{J}$ once again, we obtain the equality \eqref{VRR3}. Trivially, \eqref{VRR4} is a consequence of  \eqref{VRR3}. Moreover, \eqref{VRR5} follows directly from the associativity of $\mu_{H}$ and the convolution invertibility of $v_{R}$ because 
\begin{align*}
(H\otimes \mu_{H})\circ	(r_{H_{J}}\otimes v_{R}^{-1})
&=
(H\otimes (\mu_{H}\circ (v_{R}^{-1}\otimes (\mu_{H}\circ (\rho_{H}\otimes (v_{R}\ast v_{R}^{-1}))))))\circ \delta_{H_{J}}
\\
&=(H\otimes(\mu_{H}\circ(v_{R}^{-1}\otimes \rho_{H}) ) )\circ \delta_{H_{J}}.
\end{align*}

Finally, if $H$ is commutative, then 
$$(\varepsilon_{H}\otimes H)\circ	r_{H_{J}}=\mu_{H}\circ (\rho_{H}\otimes (v_{R}^{-1}\ast v_{R}))=\rho_{H}$$
holds by the naturality of $c$, the convolution invertibility of $v_{R}$ and the unit properties. 
	
\end{proof}

As in the previous lemmas we have a left version for Lemma \ref{VRR}. 

\begin{Lemma}
	\label{VLL}
	Let $H$ be a non-coassociative bimonoid with left codivision $l_{H}$ and let $J$ be a twist. Put $\lambda_{H}=(H\otimes \varepsilon_{H})\circ l_{H}$ and assume that  $\lambda_{H}\ast id_{H}=\varepsilon_{H}\otimes \eta_{H}$. Let $v_{L} $ and $v_{L}^{-1} $ be the morphisms defined in Lemma \ref{VL}. Define the morphism $l_{H_{J}}: H\rightarrow H\otimes H$ by 
	$$l_{H_{J}}=((\mu_{H}\circ (v_{L}\otimes (\mu_{H}\circ (\lambda_{H}\otimes v_{L}^{-1})))\otimes H)\circ \delta_{H_{J}},$$
	where $\delta_{H_{J}}$ is the coproduct defined in Theorem \ref{firstmain}.
	
	Then,
	\begin{equation*}
		l_{H_{J}}\circ \mu_{H}=((\mu_{H}\circ c_{H,H})\otimes \mu_{H})\circ (H\otimes c_{H,H}\otimes H)\circ (l_{H_{J}}\otimes l_{H_{J}})
	\end{equation*}
	holds. Moreover, 
	\begin{equation*}
		( \varepsilon_{H}\otimes H)\circ	l_{H_{J}}=id_{H}, 
	\end{equation*}
	and 
	\begin{equation*}
		(H\otimes \varepsilon_{H})\circ	l_{H_{J}}=\mu_{H}\circ (v_{L}\otimes (\mu_{H}\circ (\lambda_{H}\otimes v_{L}^{-1})))
	\end{equation*}
	also holds. As a consequence, we have the following equalities:
	\begin{equation*}
		(((H\otimes \varepsilon_{H})\circ	l_{H_{J}})\otimes H)\circ \delta_{H_{J}} =l_{H_{J}}
	\end{equation*}
	and 
	\begin{equation*}
		(\mu_{H}\otimes H)\circ	(v_{L}^{-1}\otimes l_{H_{J}} )=((\mu_{H}\circ (\lambda_{H}\otimes v_{L}^{-1}))\otimes H)\circ \delta_{H_{J}}.
	\end{equation*}
	
	Finally, if $H$ is commutative, then  $\lambda_{H}=(H\otimes \varepsilon_{H})\circ l_{H_{J}}$.
	
\end{Lemma}

\begin{Lemma}
	\label{VRBI}
	Let $H$ be a non-coassociative bimonoid with right codivision $r_{H}$, and let $J$ be a twist. Put $\rho_{H}=(\varepsilon_{H}\otimes H)\circ r_{H}$ and assume that  $id_{H}\ast \rho_{H}=\varepsilon_{H}\otimes \eta_{H}$. Let $v_{R} $ and $v_{R}^{-1} $ be the morphisms defined in Lemma \ref{VR} and $\delta_{H_{J}}$ the coproduct defined in Theorem \ref{firstmain}. Then, the equalities 
	\begin{equation}
	\label{VRBI1}
	(H\otimes \mu_{H})\circ (\delta_{H_{J}}\otimes ((\varepsilon_{H}\otimes H)\circ r_{H_{J}}))\circ J=(H\otimes \mu_{H})\circ (J^{-1}\otimes v_{R}), 
	\end{equation}
\begin{equation}
	\label{VRBI2}
	(H\otimes (\mu_{H}\circ (\rho_{H}\otimes v_{R})))\circ J^{-1}=(H\otimes \mu_{H})\circ (((H\otimes \rho_{H})\circ \delta_{H})\otimes H)\circ J, 
\end{equation}
and 
	\begin{equation}
	\label{VRBI3}
	(H\otimes \mu_{H})\circ (((H\otimes ((\varepsilon_{H}\otimes H)\circ	r_{H_{J}}))\circ \delta_{H_{J}})\otimes H)\circ J^{-1}=(H\otimes (\mu_{H}\circ (v_{R}^{-1}\otimes \rho_{H})))\circ J
\end{equation}
hold.
\end{Lemma}

\begin{proof}  We begin by establishing identity \eqref{VRBI1}:
	
\begin{itemize}
\itemindent=-22pt
	
\item[]$\hspace{0.38cm}	(H\otimes \mu_{H})\circ (\delta_{H_{J}}\otimes ((\varepsilon_{H}\otimes H)\circ r_{H_{J}}))\circ J$

\item[]$=(H\otimes \mu_{H}) \circ  ((\mu_{H\otimes H}\circ (\mu_{H\otimes H}\otimes H\otimes H)\circ (J^{-1}\otimes \delta_{H}\otimes J))\otimes (\mu_{H}\circ ((\mu_{H}\circ (v_{R}^{-1}\otimes \rho_{H})\otimes v_{R}))))\circ J$ 
\item[]$\hspace{0.38cm}	${\footnotesize (by  \eqref{VRR3})}

\item[]$=(\mu_{H}\otimes \mu_{H}) \circ (H\otimes \mu_{H\otimes H}\otimes\mu_{H})\circ (H\otimes c_{H,H}\otimes c_{H,H}\otimes H\otimes (\mu_{H}\circ (\rho_{H}\otimes (\mu_{H}\circ (\rho_{H}\otimes v_{R})))))$
\item[]$\hspace{0.38cm}	\circ (J^{-1}\otimes \delta_{H}\otimes ((H\otimes \mu_{H}\otimes H)\circ (J\otimes J^{-1}))\otimes H)\circ J$ {\footnotesize (by  the naturality of $c$ and the associativity }
\item[]$\hspace{0.38cm}	$ {\footnotesize of $\mu_{H}$)}

\item[]$=(\mu_{H}\otimes \mu_{H}) \circ (H\otimes \mu_{H\otimes H}\otimes\mu_{H})\circ (H\otimes c_{H,H}\otimes c_{H,H}\otimes H\otimes (\mu_{H}\circ (\rho_{H}\otimes (\mu_{H}\circ (\rho_{H}\otimes v_{R})))))$
\item[]$\hspace{0.38cm}	\circ (J^{-1}\otimes \delta_{H}\otimes (\partial_{1}(J)\ast \partial_{4}(J^{-1}))\otimes H)\circ J$ {\footnotesize (by  the convolution product)}

\item[]$=(\mu_{H}\otimes \mu_{H}) \circ (H\otimes \mu_{H\otimes H}\otimes\mu_{H})\circ (H\otimes c_{H,H}\otimes c_{H,H}\otimes H\otimes (\mu_{H}\circ (\rho_{H}\otimes (\mu_{H}\circ (\rho_{H}\otimes v_{R})))))$
\item[]$\hspace{0.38cm}	\circ (J^{-1}\otimes \delta_{H}\otimes (\partial_{3}(J^{-1})\ast \partial_{2}(J))\otimes H)\circ J$ {\footnotesize (by  \eqref{twist-def1})}

\item[]$=(H\otimes (\mu_{H}\circ (H\otimes \mu_{H}))) \circ ((\mu_{H\otimes H}\circ (H\otimes H\otimes c_{H,H}))\otimes H\otimes \mu_{H})$
\item[]$\hspace{0.38cm}\circ (((\mu_{H}\otimes H\otimes H)\circ (H\otimes c_{H,H}\otimes H)\circ (J^{-1}\otimes \delta_{H}))$
\item[]$\hspace{0.38cm}\otimes ((\mu_{H}\otimes \mu_{H}\otimes (\rho_{H}\circ \mu_{H}))\circ (H\otimes c_{H,H}\otimes c_{H,H}\otimes H)\circ (\delta_{H}\otimes c_{H,H}\otimes \delta_{H})\circ (J^{-1}\otimes J))\otimes (\mu_{H}\circ (\rho_{H}\otimes v_{R})))$
\item[]$\hspace{0.38cm}\circ J$ {\footnotesize (by  the  convolution product)}

\item[]$=(H\otimes (\mu_{H}\circ (H\otimes \mu_{H}))) \circ  ((\mu_{H\otimes H}\circ (H\otimes H\otimes c_{H,H}))$
\item[]$\hspace{0.38cm}\otimes (\mu_{H}\circ ((\mu_{H}\circ (H\otimes \rho_{H}))\otimes \rho_{H})\circ (\mu_{H}\otimes c_{H,H})\circ (H\otimes c_{H,H}\otimes H)\circ (H\otimes H\otimes \delta_{H}))\otimes H)$
\item[]$\hspace{0.38cm}\circ (((\mu_{H}\otimes H\otimes H)\circ (H\otimes c_{H,H}\otimes H)\circ (J^{-1}\otimes \delta_{H}))$
\item[]$\hspace{0.38cm}\otimes 
((\mu_{H}\otimes H\otimes H\otimes H)\circ (H\otimes c_{H,H}\otimes H\otimes H)\circ (\delta_{H}\otimes c_{H,H}\otimes H)\circ (J^{-1}\otimes J))\otimes (\mu_{H}\circ (\rho_{H}\otimes v_{R})))$
\item[]$\hspace{0.38cm}\circ J$ {\footnotesize (by \eqref{gr10} and  the naturality of $c$)}

\item[]$=(H\otimes (\mu_{H}\circ (H\otimes \mu_{H}))) \circ  ((\mu_{H\otimes H}\circ (H\otimes H\otimes c_{H,H}))\otimes H\otimes (\mu_{H}\circ (\rho_{H}\otimes H))) $ 
\item[]$\hspace{0.38cm}\circ  (((\mu_{H}\otimes H\otimes H)\circ (H\otimes c_{H,H}\otimes H)\circ (J^{-1}\otimes \delta_{H}))$
\item[]$\hspace{0.38cm} \otimes ((\mu_{H}\otimes (\mu_{H}\circ (H\otimes (id_{H}\ast \rho_{H}))\otimes H)\circ (H\otimes c_{H,H}\otimes c_{H,H})\circ (\delta_{H}\otimes c_{H,H}\otimes H)\circ (J^{-1}\otimes J))$
\item[]$\hspace{0.38cm}\otimes (\mu_{H}\circ (\rho_{H}\otimes v_{R})))\circ J$ {\footnotesize (by  the naturality of $c$ and the associativity of $\mu_{H}$)}

\item[]$=(H\otimes (\mu_{H}\circ (H\otimes \mu_{H})))$
\item[]$\hspace{0.38cm} \circ  (((\mu_{H\otimes H}\otimes H)\circ (\mu_{H}\otimes H\otimes H\otimes H\otimes H)\circ (H\otimes c_{H,H}\otimes H\otimes H\otimes H)\circ (J^{-1}\otimes \delta_{H\otimes H}))$
\item[]$\hspace{0.38cm}\otimes (\mu_{H}\circ ((\mu_{H}\circ (\rho_{H}\otimes \rho_{H}))\otimes v_{R})))$
\item[]$\hspace{0.38cm}\circ (H\otimes J^{-1}\otimes H)\circ J$
{\footnotesize (by  $id_{H}\ast \rho_{H}=\varepsilon_{H}\otimes \eta_{H}$, the conormal condition, and the unit properties)}

\item[]$=(H\otimes \mu_{H})\circ ((\mu_{H\otimes H}\circ (J^{-1}\otimes (\mu_{H\otimes H}\circ (\delta_{H}\otimes \delta_{H}))))\otimes  (\mu_{H}\circ ((\mu_{H}\circ (\rho_{H}\otimes \rho_{H}))\otimes v_{R})))$
\item[]$\hspace{0.38cm}\circ (H\otimes J^{-1}\otimes H)\circ J$ {\footnotesize (by  the naturality of $c$ and the associativity of $\mu_{H}$)}

\item[]$=(H\otimes \mu_{H})\circ ((\mu_{H\otimes H}\circ (J^{-1}\otimes (\delta_{H}\circ \mu_{H})))\otimes  (\mu_{H}\circ ((\mu_{H}\circ (\rho_{H}\otimes \rho_{H}))\otimes v_{R})))\circ (H\otimes J^{-1}\otimes H)\circ J$
\item[]$\hspace{0.38cm}$ {\footnotesize (by  \eqref{mu-ec})}

\item[]$=(H\otimes \mu_{H})\circ ((\mu_{H\otimes H}\circ (J^{-1}\otimes (\delta_{H}\circ \mu_{H})))\otimes  (\mu_{H}\circ ((\mu_{H}\circ (\rho_{H}\otimes \rho_{H})\circ c_{H,H})\otimes v_{R})))\circ (H\otimes c_{H,H}\otimes H)$
\item[]$\hspace{0.38cm}\circ  (J\otimes J^{-1})$ {\footnotesize (by  the naturality of $c$)}

\item[]$=(H\otimes \mu_{H})\circ ((\mu_{H\otimes H}\circ (J^{-1}\otimes \delta_{H}))\otimes  (\mu_{H}\circ (\rho_{H}\otimes v_{R})))\circ  (J\ast  J^{-1})$ {\footnotesize (by  \eqref{gr10})}

\item[]$=(H\otimes \mu_{H})\circ (J^{-1}\otimes v_{R})$ {\footnotesize (by  $J\ast J^{-1}=\eta_{H}\otimes \eta_{H}$, \eqref{gr8}, \eqref{eta-ec}, the naturality of $c$, and the unit properties).}

\end{itemize}	

Next, we prove that \eqref{VRBI2} holds. Indeed: 
	
\begin{itemize}
	\itemindent=-22pt
	\item[]$\hspace{0.38cm}	(H\otimes (\mu_{H}\circ (\rho_{H}\otimes v_{R})))\circ J^{-1}$
	
	\item[]$=(H\otimes (\mu_{H}\circ ((\mu_{H}\circ ( \rho_{H}\otimes \rho_{H}))\otimes H)))\circ (J^{-1}\otimes J)$ {\footnotesize (by  the associativity of $\mu_{H}$)}
	
	\item[]$=(H\otimes (\mu_{H}\circ (\rho_{H}\otimes H)))\circ (H\otimes (\mu_{H}\circ c_{H,H})\otimes H)\circ (J^{-1}\otimes J)$ {\footnotesize (by  \eqref{gr10} and the symmetric character}
	\item[]$\hspace{0.38cm}$ {\footnotesize of ${\sf C}$)}
	
	\item[]$=(H\otimes (\mu_{H}\circ (\rho_{H}\otimes H)))\circ (\partial_{4}(J)\ast \partial_{1}(J^{-1}))$ {\footnotesize (by  the convolution product)}
	
	\item[]$=(H\otimes (\mu_{H}\circ (\rho_{H}\otimes H)))\circ (\partial_{2}(J^{-1})\ast \partial_{3}(J))$ {\footnotesize (by  \eqref{twist-def3})}

	\item[]$=(\mu_{H}\otimes (\mu_{H}\circ ((\rho_{H}\circ \mu_{H})\otimes \mu_{H})))\circ (H\otimes ((c_{H,H}\otimes c_{H,H})\circ \delta_{H\otimes H})\otimes H)\circ (J^{-1}\otimes J) $ {\footnotesize (by the}
	\item[]$\hspace{0.38cm}$ {\footnotesize  convolution product)}
	
	\item[]$=(H\otimes \mu_{H})\circ (((\mu_{H}\otimes \rho_{H})\circ (H\otimes \delta_{H}))\otimes (\mu_{H}\circ ((\rho_{H}\ast id_{H})\otimes H)))\circ (H\otimes c_{H,H}\otimes H)\circ (J^{-1}\otimes J)$
	\item[]$\hspace{0.38cm} $ {\footnotesize (by  \eqref{gr10}), the naturality of $c$, and the associativity of $\mu_{H}$)}

	\item[]$=(H\otimes \mu_{H})\circ (((H\otimes \rho_{H})\circ \delta_{H})\otimes H)\circ J$ {\footnotesize (by  \eqref{gr5}, the conormal condition of $J^{-1}$, and the properties of}
	\item[]$\hspace{0.38cm}$ {\footnotesize  the unit).}
	
\end{itemize}

To conclude, we establish the  identity  \eqref{VRBI3} as follows:

\begin{itemize}
	\itemindent=-22pt
	\item[]$\hspace{0.38cm}	(H\otimes \mu_{H})\circ (((H\otimes ((\varepsilon_{H}\otimes H)\circ	r_{H_{J}}))\circ \delta_{H_{J}})\otimes H)\circ J^{-1}$
	
	\item[]$= (\mu_{H}\otimes \mu_{H})\circ (H\otimes ((H\otimes (\mu_{H}\circ (v_{R}^{-1}\otimes (\mu_{H}\circ c_{H,H}))))\circ (c_{H,H}\otimes \rho_{H})\circ (H\otimes J))\otimes H)$ 
	\item[]$\hspace{0.38cm} \circ (((\mu_{H}\otimes (\mu_{H}\circ c_{H,H}\circ (\rho_{H}\otimes \rho_{H})))\circ (H\otimes c_{H,H}\otimes H)\circ (J^{-1}\otimes \delta_{H}))\otimes (\mu_{H}\circ (v_{R}\otimes H)))\circ J^{-1}$
	\item[]$\hspace{0.38cm}$ {\footnotesize (by the naturality of $c$, the associativity of $\mu_{H}$, and \eqref{gr10})}
	
	\item[]$= (\mu_{H}\otimes \mu_{H})\circ (H\otimes ((H\otimes (\mu_{H}\circ (\mu_{H}\otimes \mu_{H})\circ (v_{R}^{-1}\otimes c_{H,H}\otimes \rho_{H})))\circ (c_{H,H}\otimes c_{H,H})\circ (H\otimes c_{H,H}\otimes \rho_{H})$
	\item[]$\hspace{0.38cm}\circ (H\otimes H\otimes J))\otimes H) \circ (((\mu_{H}\otimes c_{H,H})\circ (H\otimes c_{H,H}\otimes \rho_{H})\circ (J^{-1}\otimes \delta_{H}))\otimes (\mu_{H}\circ (v_{R}\otimes H)))\circ J^{-1}$ \item[]$\hspace{0.38cm}${\footnotesize (by the naturality of $c$ and the associativity of $\mu_{H}$)}
	
	\item[]$=(\mu_{H}\otimes \mu_{H})\circ (H\otimes ((H\otimes (\mu_{H}\circ ((\mu_{H}\circ (v_{R}^{-1}\otimes \rho_{H}))\otimes (\mu_{H}\circ (H\otimes \rho_{H})))))\circ (J\otimes H\otimes H))\otimes H)  $ 
	\item[]$\hspace{0.38cm}\circ (((\mu_{H}\otimes c_{H,H})\circ (H\otimes c_{H,H}\otimes \rho_{H})\circ (J^{-1}\otimes \delta_{H}))\otimes  (\mu_{H}\circ (v_{R}\otimes H)))\circ J^{-1}$ {\footnotesize (by the naturality }
	\item[]$\hspace{0.38cm}$ {\footnotesize of $c$ and the associativity of $\mu_{H}$)}
	
	\item[]$=(\mu_{H}\otimes \mu_{H})\circ (H\otimes ((H\otimes \mu_{H})\circ (((\mu_{H}\otimes (\mu_{H}\circ (v_{R}^{-1}\otimes \rho_{H})))\circ (H\otimes J))\otimes (\mu_{H}\circ (\rho_{H}\otimes \rho_{H})))$
	\item[]$\hspace{0.38cm}\circ (\delta_{H}\otimes H)\circ c_{H,H})\otimes (\mu_{H}\circ (v_{R}\otimes H)))\circ (J^{-1}\otimes J^{-1})${\footnotesize (by the naturality of $c$ and the associativity of}
	\item[]$\hspace{0.38cm}$ {\footnotesize $\mu_{H}$)}

	\item[]$= (\mu_{H}\otimes \mu_{H})\circ (H\otimes ((H\otimes \mu_{H})\circ (\mu_{H}\otimes v_{R}^{-1}\otimes (\mu_{H}\circ (\rho_{H}\otimes \rho_{H})))\circ (H\otimes J\otimes H)\circ \delta_{H})\otimes \mu_{H})$
	\item[]$\hspace{0.38cm}\circ (H\otimes H\otimes (\mu_{H}\circ (\rho_{H}\otimes v_{R}))\otimes H)\circ (H\otimes c_{H,H}\otimes H)\circ  (J^{-1}\otimes J^{-1})$ {\footnotesize (by  the associativity of $\mu_{H}$)}

	\item[]$= (\mu_{H}\otimes \mu_{H})\circ (H\otimes ((H\otimes \mu_{H})\circ (\mu_{H}\otimes v_{R}^{-1}\otimes (\mu_{H}\circ (\rho_{H}\otimes \rho_{H})))\circ (H\otimes J\otimes H)\circ \delta_{H})\otimes \mu_{H})$
	\item[]$\hspace{0.38cm}\circ  (H\otimes c_{H,H}\otimes H)\circ  (((H\otimes (\mu_{H}\circ (\rho_{H}\otimes v_{R})))\circ J^{-1})\otimes J^{-1})$
	{\footnotesize (by the naturality of $c$)}

   \item[]$= (\mu_{H}\otimes \mu_{H})\circ (H\otimes ((H\otimes \mu_{H})\circ (\mu_{H}\otimes v_{R}^{-1}\otimes (\mu_{H}\circ (\rho_{H}\otimes \rho_{H})))\circ (H\otimes J\otimes H)\circ \delta_{H})\otimes \mu_{H})$
   \item[]$\hspace{0.38cm}\circ  (H\otimes c_{H,H}\otimes H)\circ  ((((H\otimes \mu_{H})\circ (((H\otimes \rho_{H})\circ \delta_{H})\otimes H)\circ J))\otimes J^{-1})$
   {\footnotesize (by \eqref{VRBI2})}

    \item[]$=(\mu_{H}\otimes (\mu_{H}\circ (H\otimes \mu_{H})))\circ 
    (H\otimes ((((H\otimes \mu_{H})\circ (\mu_{H}\otimes ((\mu_{H}\circ (v_{R}^{-1}\otimes \rho_{H})) $ 
    \item[]$\hspace{0.38cm}\otimes (\mu_{H}\circ (\rho_{H}\otimes \rho_{H})\circ c_{H,H})))\circ (H\otimes J\otimes H\otimes H))\otimes H)\circ  (c_{H,H}\otimes c_{H,H})\circ (H\otimes c_{H,H}\otimes H))\otimes H)$
     \item[]$\hspace{0.38cm}\circ (((\delta_{H}\otimes H)\circ J)\otimes ((\delta_{H}\otimes H)\circ J^{-1}))$
    {\footnotesize (by the naturality of $c$ and the associativity of $\mu_{H}$)}

    \item[]$=((\mu_{H}\circ (\mu_{H}\otimes H))\otimes (\mu_{H}\circ (H\otimes \mu_{H})))$
    \item[]$\hspace{0.38cm}\circ 
    (H\otimes H\otimes ((H\otimes \mu_{H})\circ (H\otimes  (\mu_{H}\circ (v_{R}^{-1}\otimes \rho_{H}))\otimes  (\rho_{H}\circ \mu_{H}))\circ (J\otimes H\otimes H)) \otimes H\otimes H)$
    \item[]$\hspace{0.38cm}\circ (H\otimes c_{H,H}\otimes c_{H,H}\otimes H)\circ 
    (H\otimes H\otimes  c_{H,H}\otimes H\otimes H)\circ (((\delta_{H}\otimes H)\circ J)\otimes ((\delta_{H}\otimes H)\circ J^{-1}))$
     \item[]$\hspace{0.38cm}${\footnotesize (by  \eqref{gr10})}

	\item[]$=(\mu_{H}\otimes (\mu_{H}\circ ((\mu_{H}\circ (v_{R}^{-1}\otimes (\mu_{H}\circ (\rho_{H}\otimes \rho_{H}))))\otimes H)))\circ  (H\otimes J\otimes H\otimes H)\circ (\delta_{H}\otimes H)\circ (J\ast J^{-1})$ 
	 \item[]$\hspace{0.38cm}${\footnotesize (by  the naturality of $c$ and \eqref{mu-ec})}

	\item[]$= (H\otimes (\mu_{H}\circ (v_{R}^{-1}\otimes \rho_{H})))\circ J$ {\footnotesize (by  \eqref{eta-ec}, \eqref{gr8} and the unit properties)}.

\end{itemize}

\end{proof}

As pointed in the preceding lemmas, Lemma  \ref{VRBI} also admits a left analogue, stated as follows:

\begin{Lemma}
	\label{VLBI}
Let $H$ be a non-coassociative bimonoid with left codivision $l_{H}$ and let $J$ be a twist. Put $\lambda_{H}=(H\otimes \varepsilon_{H})\circ l_{H}$ and assume that  $\lambda_{H}\ast id_{H}=\varepsilon_{H}\otimes \eta_{H}$. Let $v_{L} $ and $v_{L}^{-1} $ be the morphisms defined in Lemma \ref{VL} and $\delta_{H_{J}}$  the coproduct defined in Theorem \ref{firstmain}.Then, the equalities 
\begin{equation}
	\label{VLBI1}
	(\mu_{H}\otimes H)\circ ((( H\otimes \varepsilon_{H})\circ l_{H_{J}})\otimes \delta_{H_{J}})\circ J^{-1}=( \mu_{H}\otimes H)\circ (v_{L}\otimes J), 
\end{equation}
\begin{equation}
	\label{VLBI2}
	((\mu_{H}\circ ( v_{L}\otimes \lambda_{H}))\otimes H)\circ J=(\mu_{H}\otimes H)\circ (H\otimes ((\lambda_{H}\otimes H)\circ \delta_{H}))\circ J^{-1}, 
\end{equation}
and 
\begin{equation}
	\label{VRLI3}
	(\mu_{H}\otimes H)\circ (H\otimes ((((H\otimes \varepsilon_{H})\circ	l_{H_{J}})\otimes H)\circ \delta_{H_{J}}))\circ J=( (\mu_{H}\circ (\lambda_{H}\otimes v_{L}^{-1}))\otimes H)\circ J^{-1}
         \end{equation}
hold.
\end{Lemma}

The following two theorems are the main results of this  subsection.

\begin{theorem}
\label{mailR3} Let 	$H$ be a right Hopf coquasigroup with right antipode $\rho_{H}$. Let $J$ be a twist and $H_{J}$ the non-coassociative bimonoid introduced in Theorem \ref{firstmain}. Then, $H_{J}$ is a right Hopf coquasigroup with right antipode 
$$\rho_{H_{J}}=\mu_{H}\circ ((\mu_{H}\circ (v_{R}^{-1}\otimes \rho_{H})\otimes v_{R})),$$
where $v_{R}$ and $v_{R}^{-1}$ are the morphisms introduced in Lemma \ref{VR}.

\end{theorem}

\begin{proof}
Since Theorem \ref{firstmain} ensures that $H_{J}$ is a non-coassociative bimonoid, it only remains to prove that the two identities in \eqref{Rhcg} hold in order to conclude that $H_{J}$ is a right Hopf coquasigroup.  We begin with the first one:

\begin{itemize}
	\itemindent=-22pt
	\item[]$\hspace{0.38cm}	(H\otimes \mu_{H})\circ (\delta_{H_{J}}\otimes \rho_{H})\circ \delta_{H_{J}}$
	
	\item[]$= (H\otimes \mu_{H})\circ ((\mu_{H\otimes H}\circ (\delta_{H_{J}}\otimes \delta_{H_{J}}))\otimes (\mu_{H}\circ ((\mu_{H}\circ (v_{R}^{-1}\otimes (\mu_{H}\circ (\rho_{H}\otimes \rho_{H})\circ c_{H,H})))\otimes v_{R})))$
	\item[]$\hspace{0.38cm}	\circ (H\otimes c_{H,H}\otimes H)\circ ((\mu_{H\otimes H}\circ (J^{-1}\otimes \delta_{H}))\otimes J)
	$ {\footnotesize (by  \eqref{gr10} and \eqref{mu-ec} for $H_{J}$)}
	
	\item[]$= (H\otimes (\mu_{H}\circ (\mu_{H}\otimes H)))\circ  ((\mu_{H\otimes H}\circ (\delta_{H_{J}}\otimes H\otimes H))\otimes H\otimes (\mu_{H}\circ (\rho_{H}\otimes v_{R})))$
	\item[]$\hspace{0.38cm}	\circ (H\otimes ((\delta_{H_{J}}\otimes (\mu_{H}\circ (v_{R}^{-1}\otimes \rho_{H})))\circ J)\otimes H)\circ  \mu_{H\otimes H}\circ (J^{-1}\otimes \delta_{H})$ {\footnotesize (by  the naturality of $c$ and the}
	\item[]$\hspace{0.38cm}${\footnotesize associativity of $\mu_{H}$)}
	
	\item[]$= (H\otimes (\mu_{H}\circ (\mu_{H}\otimes H)))\circ  ((\mu_{H\otimes H}\circ (\delta_{H_{J}}\otimes H\otimes H))\otimes H\otimes (\mu_{H}\circ (\rho_{H}\otimes v_{R})))$
	\item[]$\hspace{0.38cm}	\circ (H\otimes ((\delta_{H_{J}}\otimes (\mu_{H}\circ ((\mu_{H}\circ (v_{R}^{-1}\otimes \rho_{H}))\otimes (v_{R}\ast v_{R}^{-1}))))\circ J)\otimes H)\circ  \mu_{H\otimes H}\circ (J^{-1}\otimes \delta_{H})$ \item[]$\hspace{0.38cm}$ {\footnotesize (by  the condition of convolution invertible of $v_{R}$ and the unit properties)}

	\item[]$= (H\otimes \mu_{H}) \circ ((\mu_{H\otimes H}\circ (\delta_{H_{J}}\otimes H\otimes H))\otimes (\mu_{H}\circ (\rho_{H}\otimes v_{R})))\circ (H\otimes H\otimes  (\mu_{H}\circ (H\otimes v_{R}^{-1}))\otimes H)$
	\item[]$\hspace{0.38cm}	\circ (H\otimes ((H\otimes \mu_{H})\circ (\delta_{H_{J}}\otimes ((\varepsilon_{H}\otimes H)\circ r_{H_{J}}))\circ J)\otimes H)\circ   \mu_{H\otimes H}\circ (J^{-1}\otimes \delta_{H})
	$ {\footnotesize (by  the associativity}
    \item[]$\hspace{0.38cm}$ {\footnotesize  of $\mu_{H}$ and \eqref{VRR3})}
	
	\item[]$= (H\otimes \mu_{H}) \circ ((\mu_{H\otimes H}\circ (\delta_{H_{J}}\otimes H\otimes H))\otimes (\mu_{H}\circ (\rho_{H}\otimes v_{R})))\circ (H\otimes H\otimes  (\mu_{H}\circ (H\otimes v_{R}^{–1}))\otimes H)$
	\item[]$\hspace{0.38cm}\circ (H\otimes ((H\otimes \mu_{H})\circ (J^{-1}\otimes v_{R}))\otimes H)\circ   \mu_{H\otimes H}\circ (J^{-1}\otimes \delta_{H})
	$ {\footnotesize (by \eqref{VRBI1})}

	\item[]$=  (H\otimes \mu_{H}) \circ (H\otimes \mu_{H}\otimes H) \circ ((\mu_{H\otimes H}\circ (\delta_{H_{J}}\otimes J^{-1}))\otimes (v_{R}\ast v_{R}^{-1})\otimes (\mu_{H}\circ (\rho_{H}\otimes v_{R})))\circ  \mu_{H\otimes H}$
	\item[]$\hspace{0.38cm}\circ (J^{-1}\otimes \delta_{H})$ {\footnotesize (by the associativity of $\mu_{H}$)}

	\item[]$=  (H\otimes \mu_{H}) \circ ((\mu_{H\otimes H}\circ (\delta_{H_{J}}\otimes J^{-1}))\otimes (\mu_{H}\circ (\rho_{H}\otimes v_{R}))) \circ  \mu_{H\otimes H}\circ (J^{-1}\otimes \delta_{H})$ {\footnotesize (by the condition}
	\item[]$\hspace{0.38cm}$ {\footnotesize  of convolution invertible of $v_{R}$ and the unit properties)}

	\item[]$=  (H\otimes \mu_{H}) \circ   ((\mu_{H\otimes H}\circ (\mu_{H\otimes H}\otimes H\otimes H)\circ ((H\otimes \mu_{H}\otimes  J\otimes J^{-1}))\otimes (\mu_{H}\circ (\rho_{H}\otimes v_{R})))$
	\item[]$\hspace{0.38cm}\circ (((\mu_{H}\otimes H\otimes H) \circ (H\otimes c_{H,H}\otimes H)\circ (J^{-1}\otimes \delta_{H}))\otimes H)\circ \mu_{H\otimes H}\circ (J^{-1}\otimes \delta_{H})$ 
	{\footnotesize (by  the  definition}
	\item[]$\hspace{0.38cm}$ {\footnotesize  of  $\delta_{H_{J}}$)}
	
	\item[]$=   (H\otimes \mu_{H}) \circ  ((\mu_{H\otimes H}\circ (J^{-1}\otimes \delta_{H}))\otimes (\mu_{H}\circ (\rho_{H}\otimes v_{R})))\circ  \mu_{H\otimes H}\circ (J^{-1}\otimes \delta_{H})$ {\footnotesize (by  \eqref{extra1})}

	\item[]$=  (H\otimes \mu_{H}) \circ ((\mu_{H\otimes H}\circ (H\otimes H\otimes \delta_{H}))\otimes  (\mu_{H}\circ ((\rho_{H}\circ\mu_{H})\otimes v_{R}))) \circ (H\otimes H\otimes c_{H,H}\otimes H)$
	\item[]$\hspace{0.38cm}\circ ((\partial_{1}(J^{-1})\ast \partial_{3}(J^{-1}))\otimes \delta_{H})$ {\footnotesize (by  \eqref{mu-ec}, the naturality of $c$, and the associativity of $\mu_{H}$)}

	\item[]$=  (H\otimes \mu_{H}) \circ ((\mu_{H\otimes H}\circ (H\otimes H\otimes \delta_{H}))\otimes  (\mu_{H}\circ ((\rho_{H}\circ \mu_{H})\otimes v_{R})))\circ (H\otimes H\otimes  c_{H,H}\otimes H)$
	\item[]$\hspace{0.38cm}\circ ((\partial_{4}(J^{-1})\ast \partial_{2}(J^{-1}))\otimes \delta_{H})$ {\footnotesize (by  \eqref{twist-def2})}

	\item[]$=  (H\otimes \mu_{H}) \circ ((\mu_{H\otimes H}\circ (H\otimes H\otimes \delta_{H}))\otimes  (\mu_{H}\circ ((\mu_{H}\circ c_{H,H}\circ (\rho_{H}\otimes \rho_{H}))\otimes v_{R})))\circ (H\otimes H\otimes  c_{H,H}\otimes H)$
	\item[]$\hspace{0.38cm}\circ ((\partial_{4}(J^{-1})\ast \partial_{2}(J^{-1}))\otimes \delta_{H})$ {\footnotesize (by  \eqref{gr10})}

	\item[]$= (H\otimes \mu_{H}) \circ (((\mu_{H\otimes H}\circ  (H\otimes H\otimes ((H\otimes \mu_{H})\circ (\delta_{H}\otimes \rho_{H})\circ \delta_{H}))))\otimes (\mu_{H}\circ (\rho_{H}\otimes v_{R})))\circ (H\otimes H\otimes c_{H,H}) $
	\item[]$\hspace{0.38cm}\circ   ((\partial_{4}(J^{-1})\ast \partial_{2}(J^{-1}))\otimes H)$ 
	{\footnotesize (by  the naturality of $c$ and the associativity of $\mu_{H}$)}
	
	\item[]$= \mu_{H\otimes H}\circ (H\otimes H\otimes H\otimes  (\mu_{H}\circ (\rho_{H}\otimes v_{R})))\circ (H\otimes H\otimes c_{H,H}) \circ   ((\partial_{4}(J^{-1})\ast \partial_{2}(J^{-1}))\otimes H)$ 
	\item[]$\hspace{0.38cm}${\footnotesize (by  \eqref{Rhcg}, the naturality of $c$, and the unit properties)}

	\item[]$=  \mu_{H\otimes H}\circ (H	\otimes (\mu_{H}\circ (H\otimes \rho_{H})\circ \mu_{H\otimes H}\circ (J^{-1}\otimes \delta_{H})) \otimes H\otimes  v_{R})\circ (J^{-1}\otimes H)$ 
	{\footnotesize (by  the naturality }
	\item[]$\hspace{0.38cm}$ {\footnotesize of $c$ and the associativity of $\mu_{H}$)}

	\item[]$=  \mu_{H\otimes H}\circ (H	\otimes v_{R}^{-1}\otimes \varepsilon_{H} \otimes v_{R})\circ (J^{-1}\otimes H)$  {\footnotesize (by  \eqref{VR1-new})}

	\item[]$= H\otimes (v_{R}^{-1}\ast v_{R})$ {\footnotesize (by the conormal condition and unit properties)}

	\item[]$=  H\otimes \eta_{H}$ {\footnotesize (by the condition of convolution invertible for $v_{R}$).}

\end{itemize}

The second identity follows by:

\begin{itemize}
		\itemindent=-22pt
	\item[]$\hspace{0.38cm}	(H\otimes \mu_{H})\circ (((H\otimes \rho_{H})\circ \delta_{H_{J}})\otimes H)\circ \delta_{H_{J}}$
	
	\item[]$=(H\otimes \mu_{H})\circ ((r_{H_{J}}\circ \mu_{H})\otimes \mu_{H}) \circ (\mu_{H}\otimes c_{H,H}\otimes \mu_{H}) \circ (H\otimes c_{H,H}\otimes c_{H,H}\otimes H)\circ (J^{-1}\otimes \delta_{H}\otimes J)$ 
	\item[]$\hspace{0.38cm}${\footnotesize (by  \eqref{VRR4})}

	\item[]$= (\mu_{H}\otimes \mu_{H}) \circ (H\otimes ((H\otimes (\mu_{H}\circ c_{H,H}))\circ (c_{H,H}\otimes H)\circ (H\otimes r_{H_{J}}))\otimes H)$
	\item[]$\hspace{0.38cm}\circ (((\mu_{H}\otimes (\mu_{H}\circ c_{H,H}))\circ (H\otimes c_{H,H}\otimes H)\circ (r_{H_{J}}\otimes r_{H_{J}}))\otimes  (\mu_{H}\circ (H\otimes \mu_{H})))\circ  (H\otimes H\otimes c_{H,H}\otimes H\otimes H)$
	\item[]$\hspace{0.38cm} \circ (H\otimes c_{H,H}\otimes c_{H,H}\otimes H)\circ (J^{-1}\otimes \delta_{H}\otimes J)$ 
	{\footnotesize (by  \eqref{VRR1})}

	\item[]$= (\mu_{H}\otimes \mu_{H}) \circ (\mu_{H}\otimes r_{H_{J}}\otimes \mu_{H}) $
	\item[]$\hspace{0.38cm}\circ (H\otimes ((H\otimes c_{H,H})\circ (H\otimes \mu_{H}\otimes H) \circ (r_{H_{J}}\otimes \mu_{H}\otimes H)\circ (c_{H,H}\otimes H\otimes H)\circ (H\otimes \delta_{H}\otimes H))\otimes H)$
	\item[]$\hspace{0.38cm} \circ (((H\otimes \mu_{H})\circ (r_{H_{J}}\otimes H)\circ J^{-1})\otimes H\otimes J)$
	{\footnotesize (by  the naturality of $c$ and the associativity of $\mu_{H}$)}
	
	\item[]$= (\mu_{H}\otimes \mu_{H}) \circ (\mu_{H}\otimes r_{H_{J}}\otimes \mu_{H}) $
	\item[]$\hspace{0.38cm}\circ (H\otimes ((H\otimes c_{H,H})\circ (H\otimes \mu_{H}\otimes H) \circ (r_{H_{J}}\otimes \mu_{H}\otimes H)\circ (c_{H,H}\otimes H\otimes H)\circ (H\otimes \delta_{H}\otimes H))\otimes H)$
	\item[]$\hspace{0.38cm} \circ (((H\otimes \varepsilon_{H}\otimes \mu_{H})\circ (H\otimes r_{H_{J}}\otimes H)\circ (\delta_{H_{J}}\otimes H)\circ J^{-1})\otimes H\otimes J)$ {\footnotesize (by  \eqref{VRR4})}

	\item[]$= (\mu_{H}\otimes \mu_{H}) \circ (\mu_{H}\otimes r_{H_{J}}\otimes \mu_{H}) $
	\item[]$\hspace{0.38cm}\circ (H\otimes ((H\otimes c_{H,H})\circ (H\otimes \mu_{H}\otimes H) \circ (r_{H_{J}}\otimes \mu_{H}\otimes H)\circ (c_{H,H}\otimes H\otimes H)\circ (H\otimes \delta_{H}\otimes H))\otimes H)$
	\item[]$\hspace{0.38cm} \circ (((H\otimes (\mu_{H}\circ (v_{R}^{-1}\otimes \rho_{H})))\circ J)\otimes H\otimes J)$ {\footnotesize (by  \eqref{VRBI3})}
	
	\item[]$=  (\mu_{H}\otimes \mu_{H}) \circ (\mu_{H}\otimes r_{H_{J}}\otimes \mu_{H}) $
	\item[]$\hspace{0.38cm}\circ (H\otimes ((H\otimes c_{H,H})\circ (H\otimes \mu_{H}\otimes H)\circ (H\otimes \mu_{H}\otimes H \otimes H )\circ (((H\otimes \mu_{H})\circ (r_{H_{J}}\otimes v_{R}^{-1}))\otimes \rho_{H}\otimes H\otimes H)$
	\item[]$\hspace{0.38cm}\circ (c_{H,H}\otimes H\otimes H))\otimes H)\circ (J\otimes \delta_{H}\otimes J)$
	{\footnotesize (by  the naturality of $c$ and the associativity of $\mu_{H}$)}
	
	\item[]$=  (\mu_{H}\otimes \mu_{H}) \circ (\mu_{H}\otimes r_{H_{J}}\otimes \mu_{H}) $
	\item[]$\hspace{0.38cm}\circ (H\otimes ((H\otimes c_{H,H})\circ (H\otimes \mu_{H}\otimes H)\circ (H\otimes \mu_{H}\otimes H \otimes H )\circ (((H\otimes (\mu_{H}\circ (v_{R}^{-1}\otimes \rho_{H})))\circ \delta_{H_{J}})$
	\item[]$\hspace{0.38cm}\otimes \rho_{H}\otimes H\otimes H)\circ (c_{H,H}\otimes H\otimes H))\otimes H)\circ (J\otimes \delta_{H}\otimes J)$
	{\footnotesize (by  \eqref{VRR5})}
	
	\item[]$=  (\mu_{H}\otimes \mu_{H}) \circ (\mu_{H}\otimes r_{H_{J}}\otimes \mu_{H}) $
	\item[]$\hspace{0.38cm}\circ (H\otimes ((H\otimes c_{H,H})\circ (H\otimes \mu_{H}\otimes H)\circ (H\otimes (\mu_{H}\circ (v_{R}^{-1}\otimes (\rho_{H}\circ\mu_{H})))\otimes H \otimes H ))$
	\item[]$\hspace{0.38cm} \circ (c_{H,H}\otimes H\otimes H\otimes H )\circ  (H\otimes \delta_{H_{J}}\otimes H\otimes H))\otimes H)  \circ (J\otimes \delta_{H}\otimes J)$ {\footnotesize (by   the naturality of $c$,   \eqref{gr10},}
	\item[]$\hspace{0.38cm}$  {\footnotesize  and the associativity of $\mu_{H}$)}

	\item[]$=  (\mu_{H}\otimes \mu_{H}) \circ (\mu_{H}\otimes ((H\otimes \mu_{H})\circ (((H\otimes \mu_{H})\circ (r_{H_{J}}\otimes v_{R}^{-1}))\otimes (\rho_{H}\circ \mu_{H}))\circ (c_{H,H}\otimes H))\otimes \mu_{H})$
	\item[]$\hspace{0.38cm}\circ (H\otimes c_{H,H}\otimes c_{H,H}\otimes H\otimes H) \circ (H\otimes H\otimes \delta_{H_{J}}\otimes c_{H,H}\otimes H)\circ (J\otimes \delta_{H}\otimes J)$ {\footnotesize (by  the naturality}
	\item[]$\hspace{0.38cm}$ {\footnotesize of $c$ and the associativity of $\mu_{H}$)}
	
	\item[]$=  (\mu_{H}\otimes \mu_{H}) \circ   (H\otimes ((((H\otimes (\mu_{H}\circ (v_{R}^{-1}\otimes \rho_{H})))\circ \delta_{H_{J}})\otimes \mu_{H})\circ (c_{H,H}\otimes H) \circ (\mu_{H}\otimes J)))$
	\item[]$\hspace{0.38cm}\circ (((H\otimes \rho_{H})\circ\mu_{H\otimes H}\circ  (J\otimes \delta_{H_{J}}))\otimes H)\circ \delta_{H}  $ 
	{\footnotesize (by  the naturality of $c$, the associativity of $\mu_{H}$, and  \eqref{VRR5})}

	\item[]$=  (\mu_{H}\otimes \mu_{H}) \circ   (H\otimes ((((H\otimes (\mu_{H}\circ (v_{R}^{-1}\otimes \rho_{H})))\circ \delta_{H_{J}})\otimes \mu_{H})\circ (c_{H,H}\otimes H) \circ (\mu_{H}\otimes J)))$
	\item[]$\hspace{0.38cm}\circ (((H\otimes \rho_{H})\circ\mu_{H\otimes H}\circ  (\delta_{H}\otimes J))\otimes H)\circ \delta_{H}  $ 
	{\footnotesize (by  \eqref{ofin1})}
	
	\item[]$=  (\mu_{H}\otimes \mu_{H}) \circ   (H\otimes ((((H\otimes (\mu_{H}\circ (v_{R}^{-1}\otimes \rho_{H})))\circ \delta_{H_{J}})\otimes \mu_{H})\circ (c_{H,H}\otimes H) \circ (\mu_{H}\otimes J)))$
	\item[]$\hspace{0.38cm}\circ (((\mu_{H}\otimes (\mu_{H}\circ (\rho_{H}\otimes \rho_{H}) \circ c_{H,H}))\circ (H\otimes c_{H,H}\otimes H) \circ  (\delta_{H}\otimes J))\otimes H)\circ \delta_{H}  $ 
	{\footnotesize (by  \eqref{gr10})}
	
	\item[]$= (\mu_{H}\otimes \mu_{H}) \circ  (H\otimes ((((H\otimes (\mu_{H}\circ (v_{R}^{-1}\otimes \rho_{H})))\circ \delta_{H_{J}})\otimes \mu_{H})\circ (c_{H,H}\otimes H)))$
	\item[]$\hspace{0.38cm}\circ (((\mu_{H}\otimes (\mu_{H}\circ ((\mu_{H}\circ (\rho_{H}\otimes \rho_{H}))\otimes H))\otimes H\otimes H)))\circ (H\otimes J\otimes H\otimes H\otimes J)\circ (\delta_{H}\otimes H)\circ \delta_{H}$
	\item[]$\hspace{0.38cm}$ {\footnotesize (by the naturality of $c$)}

	\item[]$= (\mu_{H}\otimes \mu_{H}) \circ  (H\otimes ((((H\otimes (\mu_{H}\circ (v_{R}^{-1}\otimes \rho_{H})))\circ \delta_{H_{J}})\otimes \mu_{H})\circ (c_{H,H}\otimes H)))$
	\item[]$\hspace{0.38cm}\circ (\mu_{H}\otimes (\mu_{H}\circ (\rho_{H}\otimes H))\otimes J)\circ (H\otimes J\otimes H)\circ 
	(H\otimes \mu_{H})\circ (((H\otimes \rho_{H})\circ \delta_{H})\otimes H)\circ \delta_{H}
	$ {\footnotesize (by the}
	\item[]$\hspace{0.38cm}$ {\footnotesize  associativity of $\mu_{H}$)}
	
	\item[]$= (\mu_{H}\otimes \mu_{H})$ 
	\item[]$\hspace{0.38cm}\circ 
	(H\otimes 
	((\mu_{H}\otimes \mu_{H})\circ (H\otimes ((H\otimes v_{R}^{-1}\otimes (\mu_{H}\circ (\rho_{H}\otimes H)))
	\circ (\delta_{H_{J}}\otimes H)\circ c_{H,H}\circ (\rho_{H}\otimes H))
	\otimes H) $
	\item[]$\hspace{0.38cm}\circ (H\otimes J\otimes J) $ 
   {\footnotesize (by  \eqref{Rhcg}, the unit and the counit properties, and the associativity of $\mu_{H}$)}
	
   \item[]$= (\mu_{H}\otimes \mu_{H})$ 
   \item[]$\hspace{0.38cm}\circ 
   (H\otimes  
   ((\mu_{H}\otimes \mu_{H})\circ (H\otimes ((H\otimes v_{R}^{-1}\otimes (\mu_{H}\circ (\rho_{H}\otimes \rho_{H})\circ c_{H,H}))\circ (c_{H,H}\otimes H)\circ (H\otimes \delta_{H_{J}}))))
   \otimes H)$
    \item[]$\hspace{0.38cm}\circ (H\otimes J\otimes J)$
    {\footnotesize (by  the naturality of $c$)}

	\item[]$=(\mu_{H}\otimes (\mu_{H}\circ ((\mu_{H}\circ (v_{R}^{-1}\otimes \rho_{H}) ) \otimes H)))\circ 
	(H\otimes (\mu_{H\otimes H}\circ (J\otimes \delta_{H_{J}}))\otimes H)\circ (H\otimes J) $ 
	{\footnotesize (by  \eqref{gr10})}
	
	\item[]$=(\mu_{H}\otimes (\mu_{H}\circ ((\mu_{H}\circ (v_{R}^{-1}\otimes \rho_{H}) ) \otimes H)))\circ 
	(H\otimes (\partial_{3}(J)\ast \partial_{1}(J))) $ 
	 {\footnotesize (by  \eqref{ofin1} and \eqref{twist-def1n})}

    \item[]$=(\mu_{H}\otimes (\mu_{H}\circ ((\mu_{H}\circ (v_{R}^{-1}\otimes \rho_{H}) ) \otimes H)))\circ 
   (H\otimes (\partial_{2}(J)\ast \partial_{4}(J))) $ 
   {\footnotesize (by  \eqref{twist-def})}

   \item[]$=(\mu_{H}\otimes (\mu_{H}\circ ((\mu_{H}\circ(v_{R}^{-1}\otimes (\mu_{H}\circ (\rho_{H}\otimes \rho_{H})\circ c_{H,H}))\otimes \mu_{H})\circ (H\otimes c_{H,H}\otimes H)\circ (\delta_{H}\otimes J)))$
   \item[]$\hspace{0.38cm}\circ (H\otimes J)$ 
   {\footnotesize (by  \eqref{twist-def1n} and \eqref{gr10})}
   
   \item[]$=(\mu_{H}\otimes (\mu_{H}\circ (\mu_{H}\otimes \mu_{H})\circ (v_{R}^{-1}\otimes H\otimes (\rho_{H}\ast id_{H})\otimes H)\circ ( c_{H,H}\otimes H)\circ (H\otimes \rho_{H}\otimes H)))\circ (H\otimes J\otimes J)$
   \item[]$\hspace{0.38cm}$ 
   {\footnotesize (by  the naturality of $c$ and the associativity of $\mu_{H}$)}
	
		\item[]$= H\otimes (v_{R}^{-1}\ast v_{R}) $ {\footnotesize (by  \eqref{Rhcg1}, the conormal condition for $J$, and the unit properties)}
	
	\item[]$= H\otimes \eta_{H} $ {\footnotesize (by  the convolution invertible condition for $v_{R}$).}
	
\end{itemize}
	
\end{proof}

The left version of the previous theorem is: 

\begin{theorem}
	\label{mailL3} Let 	$H$ be a left Hopf coquasigroup with left antipode $\lambda_{H}$. Let $J$ be a twist and  $H_{J}$ the non-coassociative bimonoid introduced in Theorem \ref{firstmain}. Then, $H_{J}$ is a left Hopf coquasigroup with left  antipode 
	$$\lambda_{H_{J}}=\mu_{H}\circ (v_{L}\otimes (\mu_{H}\circ (\lambda_{H}\otimes v_{L}^{-1}))),$$
	where $v_{L}$ and $v_{L}^{-1}$ are the  morphisms defined in Lemma \ref{VL}.
\end{theorem}

The, as a consequence of Theorems \ref{mailR3} and \ref{mailL3} we obtain the following corollary:

\begin{Corollary}
	\label{mailRL3} Let $H$ be a Hopf coquasigroup with antipode $S_{H}$. Let $J$ be a twist and  $H_{J}$ the non-coassociative bimonoid introduced in Theorem \ref{firstmain}. Then, $H_{J}$ is a  Hopf coquasigroup.
\end{Corollary}

\begin{proof} The proof follows directly from Theorems \ref{mailR3} and \ref{mailL3}. Note that in  this case 
	$\rho_{H}=\lambda_{H}=S_{H}$ because a Hopf coquasigroup with antipode $S_{H}$ is a right Hopf coquasigroup with antipode $\rho_{H}=S_{H}$ and a left Hopf coquasigroup with antipode  $\lambda_{H}=S_{H}$.  As a consequence, $v_{R}=v_{L}^{-1}$ and then
	$$S_{H_{J}}=\rho_{H_{J}}=\lambda_{H_{J}}.$$
\end{proof}

\subsection{{\sc  Examples}} 

In this subsection we will work in the category of vector spaces over a field ${\mathbb K}$ as in Example \ref{ex1}. We now apply the general deformation procedure  to the Hopf coquasigroup $D={\mathbb K}[{\sf S}^7]\rtimes {\mathbb K}G$ 
constructed in the quoted example. As mentioned in the introduction, this example was presented in \cite{KM} as the first quantum Hopf coquasigroup and was conjectured to be the first of a broader family of examples. The following construction shows that twist deformation naturally generates such additional examples.

\begin{Lemma}
\label{bullp}
In the conditions of Example \ref{ex1}, for $r\in G$,  and 
$$p_{0}=\frac{1}{2}(1_{D}+T_{1}^r), \;\; p_{1}=\frac{1}{2}(1_{D}-T_{1}^r),$$
we have the following:
\begin{itemize}
\item[(i)] $T_{1}^r\bullet T_{1}^r=1_{D}$.
\item[(ii)] $p_{0}+p_{1}=1_{D}$.
\item[(iii)] $p_{0}\bullet T_{1}^r=T_{1}^r\bullet p_{0}=p_{0}$. 
\item[(iv)] $p_{1}\bullet T_{1}^r=T_{1}^r\bullet p_{1}=-p_{1}$. 
\item[(v)] $p_{0}\bullet p_{0}=p_{0}$, $p_{1}\bullet p_{1}=p_{1}$, $p_{0}\bullet p_{1}=p_{1}\bullet p_{0}=0$
\item[(vi)] $p_{0}\bullet T_{c}^u=\displaystyle\frac{1}{2}(T_{c}^u+\chi(c,r)T_{c}^{u+r})$, $ T_{c}^u\bullet p_{0}=\displaystyle\frac{1}{2}(T_{c}^u+T_{c}^{u+r}).$
\item[(vii)] $p_{1}\bullet T_{c}^u=\displaystyle\frac{1}{2}(T_{c}^u-\chi(c,r)T_{c}^{u+r})$, $ T_{c}^u\bullet p_{1}=\displaystyle\frac{1}{2}(T_{c}^u-T_{c}^{u+r}).$
\item[(viii)] $$p_{0}\bullet T_{c}^u\bullet p_{0}=\frac{1+\chi(c,r)}{4}(T_{c}^u+T_{c}^{u+r})=\left\{ \begin{array}{l}
\displaystyle\frac{1}{2}(T_{c}^u+T_{c}^{u+r}) \;\; {\rm if} \;\; rc=0,\\
	\\
	0 \;\; {\rm if} \;\; rc=1,
\end{array}\right.$$
$$p_{1}\bullet T_{c}^u\bullet p_{0}=\frac{1-\chi(c,r)}{4}(T_{c}^u+T_{c}^{u+r})=\left\{ \begin{array}{l}
	 0\;\; {\rm if} \;\; rc=0,\\
	\\
	\displaystyle\frac{1}{2}(T_{c}^u+T_{c}^{u+r}) \;\; {\rm if} \;\; rc=1,
\end{array}\right.$$
$$p_{0}\bullet T_{c}^u\bullet p_{1}=\frac{1-\chi(c,r)}{4}(T_{c}^u-T_{c}^{u+r})=\left\{ \begin{array}{l}
	0\;\; {\rm if} \;\; rc=0,\\
	\\
	\displaystyle\frac{1}{2}(T_{c}^u-T_{c}^{u+r}) \;\; {\rm if} \;\; rc=1,
\end{array}\right.$$
$$p_{1}\bullet T_{c}^u\bullet p_{1}=\frac{1+\chi(c,r)}{4}(T_{c}^u-T_{c}^{u+r})=\left\{ \begin{array}{l}
	\displaystyle\frac{1}{2}(T_{c}^u-T_{c}^{u+r})\;\; {\rm if} \;\; rc=0,\\
	\\
	0 \;\; {\rm if} \;\; rc=1.
\end{array}\right.$$
\item[(ix)] $\varepsilon_{D}(p_{0})=1$, $\varepsilon_{D}(p_{1})=0.$
\item[(x)] $\delta_{D}(p_{0})=p_{0}\otimes p_{0}+p_{1}\otimes p_{1}$, $\delta_{D}(p_{1})=p_{0}\otimes p_{1}+p_{1}\otimes p_{0}$.
\end{itemize}
\end{Lemma}

\begin{proof}
In this lemma (i) follows by (\ref{muD1}) and (ii) from the definition of $p_{0}$ and $p_{1}$. The equalities in (iii) and (iv) follows by (i), and (v) is a consequence of (iii) and (iv). The proof for (vi) and (vii) follows from (\ref{muD2})  and (\ref{muD3}). Using (vi) and (vii)  we prove (viii). Finally, (ix) follows by (\ref{varepD}) and (x), by (\ref{copD1}).
\end{proof}

\begin{theorem} 
\label{mainex}
In the conditions of Lemma \ref{bullp}, 
$$J_{r}=p_{0}\otimes 1_{D}+p_{1}\otimes T_{1}^{r}$$
is a twist for $D={\mathbb K}[{\sf S}^7]\rtimes {\mathbb K}G$ such that $J_{r}^{-1}=J_{r}$.
\end{theorem}

\begin{proof}
We first verify the fact that $J_{r}$ is conormal. By  (ix) of Lemma  \ref{bullp} and \eqref{varepD}, 
	\[
	(\varepsilon_D\otimes id_D)(J_r)
	=
	1_D
	=
	(id_D\otimes\varepsilon_D)(J_r).
	\]

It remains to prove the twist identity 
	\begin{equation}
		\label{twistex}
		(\delta_D\otimes id_D)(J_r)\diamond(J_r\otimes1_D)
		=
		(id_D\otimes\delta_D)(J_r)\diamond(1_D\otimes J_r),
	\end{equation}
where $\diamond$ denotes the product on $D\otimes D\otimes D$.
	
To establish \eqref{twistex}, we begin by computing the two terms appearing in the identity.
	
Using part (x) of Lemma \ref{bullp}, we obtain
	\[
	(\delta_D\otimes id_D)(J_r)
	=
	p_0\otimes p_0\otimes1_D
	+
	p_1\otimes p_1\otimes1_D
	+
	p_0\otimes p_1\otimes T_1^r
	+
	p_1\otimes p_0\otimes T_1^r,
	\]
whereas \eqref{copD1} together with the unit properties gives
	\[
	(id_D\otimes\delta_D)(J_r)
	=
	p_0\otimes1_D\otimes1_D
	+
	p_1\otimes T_1^r\otimes T_1^r.
	\]
	
Now, applying parts (iii)--(v) of Lemma  \ref{bullp}, a direct computation yields
\begin{align*}
		(\delta_D\otimes id_D)(J_r)
		\diamond
		(J_r\otimes1_D)
		&=
		(p_0\otimes p_0\otimes1_D
		+p_1\otimes p_1\otimes1_D
		\\
		&
		+p_0\otimes p_1\otimes T_1^r
		+p_1\otimes p_0\otimes T_1^r)
		\\
		&\diamond
		(p_0\otimes1_D\otimes1_D
		+p_1\otimes T_1^r\otimes1_D)
		\\
		&=
		p_0\otimes p_0\otimes1_D
		-
		p_1\otimes p_1\otimes1_D
		\\
		&
		+p_0\otimes p_1\otimes T_1^r
		+p_1\otimes p_0\otimes T_1^r.
\end{align*}
	
On the other hand,
\begin{align*}
		(id_D\otimes\delta_D)(J_r)
		\diamond
		(1_D\otimes J_r)
		&=
		(p_0\otimes1_D\otimes1_D
		+p_1\otimes T_1^r\otimes T_1^r)
		\\
		&\diamond
		(1_D\otimes p_0\otimes1_D
		+1_D\otimes p_1\otimes T_1^r)
		\\
		&=
		p_0\otimes p_0\otimes1_D
		-
		p_1\otimes p_1\otimes1_D
		\\
		&
		+p_0\otimes p_1\otimes T_1^r
		+p_1\otimes p_0\otimes T_1^r.
\end{align*}
	
Hence both sides coincide, proving that $J_r$ satisfies the twist equation.
	
To conclude, we prove that $J_r$ is self-inverse. By parts (i), (ii), and (v) of Lemma \ref{bullp},
\[
	J_r\ast J_r
	=
	p_0\otimes1_D
	+
	p_1\otimes1_D
	=
	(p_0+p_1)\otimes1_D
	=
	1_D\otimes1_D.
\]
Therefore,
\[
	J_r^{-1}=J_r.
\]
	
This completes the proof.
\end{proof}

Consequently, by   Theorem \ref{mainex} and the general theory of the previous sections we have a twisted deformation of $D$, denoted by $D_{J_{r}}$, where 
$$1_{D_{J_{r}}}=1_{D},\;\; \mu_{D_{J_{r}}}=\mu_{D}, \;\; \varepsilon_{D_{J_{r}}}=\varepsilon_{D}$$
and 
$$ \delta_{D_{J_{r}}}=\mu_{D\otimes D\otimes D}\circ (J_{r}\otimes \delta_{D}\otimes J_{r}), \;\; S_{D_{J_{r}}}=v_{J_{r}}^{-1}\bullet S_{D}\bullet v_{J}$$
for $v_{J_{r}}=T_{1}^r=v_{J_{r}}^{-1}$.

By (viii) of Lemma \ref{bullp}, we obtain an  explicit formulation of the coproduct $ \delta_{D_{J_{r}}}$:
\begin{align*}
	\;& \delta_{D_{J_{r}}}(T_{c}^u) \nonumber \\
	\; &= \sum_{a+b=c} F(a,b)[  p_{0}\bullet T_{a}^{u}\bullet p_{0}\otimes T_{b}^{u} + p_{1}\bullet T_{a}^{u}\bullet p_{0} \otimes \chi(b,r)T_{b}^{u+r}  +  p_{0}\bullet T_{a}^{u}\bullet p_{1} \otimes T_{b}^{u+r} \\
	 \; &\hspace{0.5cm}+  p_{1}\bullet T_{a}^{u}\bullet p_{1} \otimes \chi(b,r)T_{b}^{u}]\\
	 \; &= \sum_{a+b=c} F(a,b)[ (\frac{1+\chi(a,r)}{4}(T_{a}^u+T_{a}^{u+r}))\otimes T_{b}^{u} +  (\frac{1-\chi(a,r)}{4}(T_{a}^u+T_{a}^{u+r})) \otimes \chi(b,r)T_{b}^{u+r}   \\
	 \; &\hspace{0.5cm}+   (\frac{1-\chi(a,r)}{4}(T_{a}^u-T_{a}^{u+r})) \otimes T_{b}^{u+r}+   (\frac{1+\chi(a,r)}{4}(T_{a}^u-T_{a}^{u+r})) \otimes \chi(b,r)T_{b}^{u}]\\
	 \; &= \sum_{a+b=c} F(a,b)[ (\frac{1+\chi(a,r)}{4}((1+\chi(b,r))T_{a}^u+ (1-\chi(b,r))T_{a}^{u+r}))\otimes T_{b}^{u}\\
	 \; &\hspace{0.5cm} + (\frac{1-\chi(a,r)}{4}((1+\chi(b,r))T_{a}^u- (1-\chi(b,r))T_{a}^{u+r}))\otimes T_{b}^{u+r}].
\end{align*}

Then, we have the following cases:
\begin{itemize}
	\item[{\sc Case 1:}] $ra=0$, i.e. $\chi(a,r)=1$.
	$$\delta_{D_{J_{r}}}(T_{c}^u)=\displaystyle \sum_{a+b=c}\frac{F(a,b)}{2}((1+\chi(b,r))T_{a}^u+ (1-\chi(b,r))T_{a}^{u+r})\otimes T_{b}^{u}.$$
	\begin{itemize}
		\item[(1.1)] $rb=0$, i.e. $\chi(b,r)=1$.
		$$\delta_{D_{J_{r}}}(T_{c}^u)=\sum_{a+b=c} F(a,b) T_{a}^u\otimes T_{b}^{u}.$$
		\item[(1.2)] $rb=1$, i.e. $\chi(b,r)=-1$.
		$$\delta_{D_{J_{r}}}(T_{c}^u)=\sum_{a+b=c} F(a,b) T_{a}^{u+r}\otimes T_{b}^{u}.$$
	\end{itemize}
		\item[{\sc Case 2:}] $ra=1$, i.e. $\chi(a,r)=-1$.
	$$\delta_{D_{J_{r}}}(T_{c}^u)=\displaystyle\sum_{a+b=c}\frac{F(a,b)}{2}((1+\chi(b,r))T_{a}^u- (1-\chi(b,r))T_{a}^{u+r})\otimes T_{b}^{u+r}.$$
	\begin{itemize}
		\item[(2.1)] $rb=0$, i.e. $\chi(b,r)=1$.
		$$\delta_{D_{J_{r}}}(T_{c}^u)=\sum_{a+b=c}F(a,b)T_{a}^u\otimes T_{b}^{u+r}.$$
		\item[(2.2)] $rb=1$, i.e. $\chi(b,r)=-1$.
		$$\delta_{D_{J_{r}}}(T_{c}^u)=-\sum_{a+b=c}F(a,b) T_{a}^{u+r}\otimes T_{b}^{u+r}.$$
	\end{itemize}
\end{itemize}

Therefore, written concisely:
$$\delta_{D_{J_{r}}}(T_{c}^u) = \sum_{a+b=c} F(a,b) \chi(a,r)^{rb}T_{a}^{u+(rb)r}\otimes T_{b}^{u+(ra)r}$$ 
and for  $T_{1}^u$ we have that 
$$\delta_{D_{J_{r}}}(T_{1}^u)=T_{1}^u\otimes T_{1}^u$$
because 
$$p_{0}\bullet T_{1}^{u+r}=p_{0}\bullet T_{1}^{u}=\displaystyle \frac{1}{2}(T_{1}^u+T_{1}^{u+r}),$$
$$p_{1}\bullet T_{1}^{u+r}=-p_{1}\bullet T_{1}^{u}=\displaystyle \frac{1}{2}(T_{1}^{u+r}-T_{1}^{u}),$$
$$ T_{1}^{u+r}\bullet p_{0}= T_{1}^{u}\bullet p_{0}=\displaystyle \frac{1}{2}(T_{1}^u+T_{1}^{u+r}),$$
and 
$$ T_{1}^{u+r}\bullet p_{1}= -T_{1}^{u}\bullet p_{1}=\displaystyle \frac{1}{2}(T_{1}^{u+r}-T_{1}^{u})$$
hold.

For the antipode $S_{D_{J_{r}}}$ we have that 
$$S_{D_{J_{r}}}(T_{c}^u) =\chi(c,u)\chi(c,r)F(c,c)T_{c}^u$$
and $S_{D_{J_{r}}}(T_{1}^u) =T_{1}^u$ also holds. 

Observe that, if $r=0$, then $J_{0}=1_{D}\otimes 1_{D}$ and $D=D_{J_{0}}$. If $r\neq 0$ we have new non-commutative and non-cocommutative Hopf coquasigroups $D_{J_{r}}$ obtained by  the deformation of $D={\mathbb K}[{\sf S}^7]\rtimes {\mathbb K}G$ by the twist $J_{r}$.

\section*{Funding}

The authors acknowledge financial support from  Ministerio de Ciencia e Innovaci\'on. Agencia Estatal de Investigaci\'on (Spain) grant no.  PID2024-15502NB-I00 (European FEDER support included, UE).  B. Ramos Pérez is funded by Xunta de Galicia through the Competitive Reference Groups (GRC) grant no. ED431C 2023/31,  and the fellowship grant no. ED481A-2023-023. Funding for open access charge: Universidade de Vigo.

%
%
%
%

\end{document}